\documentclass[11pt]{article}
\usepackage{latexsym}
\usepackage{amsmath}
\usepackage{amsthm,color}
\usepackage{amssymb}
\usepackage{mathrsfs}
\usepackage{picinpar,graphicx}
\usepackage[colorlinks,  linkcolor=blue,  anchorcolor=blue, citecolor=blue]{hyperref}

\usepackage{lscape}
\newtheorem{theorem}{\indent Theorem}[section]

\newtheorem{lemma}[theorem]{\indent Lemma}
\newtheorem{remark}{\indent Remark}[section]

\begin{document}
\renewcommand{\baselinestretch}{1.3}


\begin{center}
    {\large \bf   Normalized solutions of linearly coupled Choquard system with potentials}
\vspace{0.5cm}\\{\sc Meng  Li$^*$}\\  
\end{center}


\renewcommand{\theequation}{\arabic{section}.\arabic{equation}}
\numberwithin{equation}{section}


\begin{abstract}
In this paper, we consider the existence of solutions for  the linearly  coupled Choquard system with potentials
\begin{align*}
\left\{\begin{aligned}
&-\Delta u+\lambda_1 u+V_1(x)u=\mu_1(I_{\alpha}\star|u|^p)|u|^{p-2}u+\beta(x) v,\\
&-\Delta v+\lambda_2 v+V_2(x)v=\mu_2(I_{\alpha}\star|v|^q)|v|^{q-2}u+\beta(x) u,
\end{aligned}
\right.\quad x\in \mathbb{R}^N,
\end{align*}
under the constraint
\begin{align*}
\int_{\mathbb{R}^N}u^2dx=\xi^2,~ \int_{\mathbb{R}^N}v^2dx=\eta^2,
\end{align*}
where $I_{\alpha}=\frac{1}{|x|^{N-\alpha}},~\alpha\in(0,N),~1+\frac{\alpha}{N}<p,~q<\frac{N+\alpha}{N-2},~\mu_1>0,~\mu_2>0$ and $\beta(x)$ is a fixed function.

\end{abstract}

\vspace{-1 cm}

\footnote[0]{ \hspace*{-7.4mm}
$^{*}$ Corresponding author.
E-mails:  mengl@hust.edu.cn}
\section{Introduction}
In this paper, we consider the existence of solutions for  the linearly  coupled Choquard system with potentials
\begin{align}\label{system}
\left\{\begin{aligned}
&-\Delta u+\lambda_1 u+V_1(x)u=\mu_1(I_{\alpha}\star|u|^p)|u|^{p-2}u+\beta(x) v,\\
&-\Delta v+\lambda_2 v+V_2(x)v=\mu_2(I_{\alpha}\star|v|^q)|v|^{q-2}u+\beta(x) u,
\end{aligned}
\right.\quad x\in \mathbb{R}^N,
\end{align}
under the constraint
\begin{align}\label{constraint}
\int_{\mathbb{R}^N}u^2dx=\xi^2,~ \int_{\mathbb{R}^N}v^2dx=\eta^2,
\end{align}
where $I_{\alpha}=\frac{1}{|x|^{N-\alpha}},~\alpha\in(0,N),~1+\frac{\alpha}{N}<p,~q<\frac{N+\alpha}{N-2},~\mu_1>0,~\mu_2>0$ and $\beta(x)$ is a fixed function.

If $\beta(x)=0$ and $V_1(x)=V_2(x)=0$, then the system \eqref{system} can be reduced to the scalar Choquard equation
\begin{align}\label{ce}
-\Delta u+\lambda u=\mu(I_{\alpha}\star|u|^p)|u|^{p-2}u.
\end{align}
Physical motivations arise from the case $N = 3$ and $\alpha= 2$. In 1954, Pekar \cite{P54} described a polaron
at rest in the quantum theory. In 1976, to model an electron trapped in its own hole, Choquard \cite{L77} considered \eqref{ce} as an approximation to Hartree-Fock theory of one-component plasma. In particular cases, Penrose \cite{P96} investigated the selfgravitational collapse of a quantum mechanical wave function. The system of weakly coupled equations has been widely considered in recent years and it has applications especially in nonlinear optics \cite{M87,M89}. Furthermore, nonlocal nonlinearities have attracted considerable interest
as a means of eliminating collapse and stabilizing multidimensional solitary waves. It appears naturally
in optical systems \cite{L66} and is known to influence the propagation of electromagnetic waves in plasmas \cite{BC00}.

In \eqref{system}, if $\lambda_1$ and $\lambda_2$ are given, some literature call it as fixed frequency problem.  There are some related articles. In \cite{Y19}, the authors
studied the existence and nonexistence results for a class of linearly coupled Choquard system \eqref{system} with $V_1(x)=V_2(x)=0$ for critical case. In \cite{XMX20},
the authors were concerned with the following linearly coupled Choquard-type system
\begin{align*}
\left\{\begin{aligned}
&-\Delta u+\lambda_1 u=\mu_1(I_{\alpha}\star F(u))F'(u)+\beta v,\\
&-\Delta v+\lambda_2 v=\mu_2(I_{\alpha}\star F(v))F'(v)+\beta u,
\end{aligned}
\right.\quad x\in \mathbb{R}^N.
\end{align*}
Under some suitable conditions on $F'$, the existence of vector solutions was proved by using variational methods. Moreover, they proved the asymptotic behavior of the solutions as $\beta\to0^{+}$.

Condition \eqref{constraint}  is called as the normalization condition, which imposes a normalization on the $L^2$-masses of $u$ and $v$. The solutions to the Schr\"odiger system \eqref{system} under the constraint \eqref{constraint} are usually referred as normalized solutions. In order to obtain the solution to the fractional Schr\"odiger system \eqref{system}  satisfying the normalization condition \eqref{constraint}, one need to consider
the critical point of the functional $E_{V}(u,v)$ on  $S_{\xi}\times S_{\eta}$(see \eqref{fctl} and \eqref{l2sc}). And then, $\lambda_1$ and $\lambda_2$ appear as Lagrange multipliers with respect to the mass constraint, which cannot be determined a priori, but are part of the unknown. Some literature called this problem as fixed mass problem.

Recently, the normalized solutions of nonlinear
Schr\"{o}dinger equations and systems has attracted many researchers, see more details \cite{BC13,BJ18,BJN16,B20,BRRV21,BZZ20,CZ21,GJ18,J97,J2008,J20,LZ21,S20,SN20,W21,YZ21}. In particular, for the normalized solutions of Choquard equations, there are many papers such as \cite{JL21,LL20,LY14,Y16,YC20}. For the normalized solutions of nonlinearly coupled Choquard system, there are many papers, for example \cite{WJ21,WY18}.

Up to our knowledge, there is no paper about the normalized solutions of linearly coupled Choquard system. In our paper, we first consider the normalized solutions of linearly coupled Choquard system \eqref{system} without potentials, and then consider the normalized solutions of linearly coupled Choquard system \eqref{system} with different potentials.

The corresponding functional of \eqref{system} is
\begin{align}\label{fctl}
\begin{split}
E_{V}(u,v)=&\frac12\int_{\mathbb{R}^N}(|\nabla u|^2+|\nabla v|^2)dx+\frac12\int_{\mathbb{R}^N}(V_1(x)|u|^2+V_2(x)|v|^2)dx
\\&-\frac{\mu_1}{2p}\int_{\mathbb{R}^N}\int_{\mathbb{R}^N}\frac{|u(x)|^{p}| u(y)|^{p}}{|x-y|^{N-\alpha}}dydx
-\frac{\mu_2}{2q}\int_{\mathbb{R}^N}\int_{\mathbb{R}^N}\frac{|u(x)|^{q}| u(y)|^{q}}{|x-y|^{N-\alpha}}dydx\\
&-\int_{\mathbb{R}^N}\beta(x)uvdx.
\end{split}
\end{align}
If $V_1(x)=V_2(x)=0$, then we denote that $E_{V}(u,v)=E(u,v)$.
Denote the $L^2$-sphere
\begin{align}\label{l2sc}
S_{\xi}=\{u\in H^1(\mathbb{R}^N):\|u\|_2^2=\xi^2\}\text{~and~}S_{\eta}=\{u\in H^1(\mathbb{R}^N):\|u\|_2^2=\eta^2\}.
\end{align}

First, we consider the linearly coupled Choquard system \eqref{system} without potentials, i.e.,  $V_1(x)=V_2(x)=0$.

If $1+\frac{\alpha}{N}<p,q<1+\frac{\alpha+2}{N}$, then the Choquard terms are $L^2$-subcritical and $E(u,v)$ is bounded from below on $S_{\xi}\times S_{\eta}$. Let $\beta=\beta(x)$ be a fixed constant, we focus on the problem
\begin{align}\label{problem}
e(\xi,\eta):=\inf_{S_{\xi}\times S_{\eta}}E(u,v)
\end{align}
\begin{theorem}\label{t1}
Suppose that $\beta>0$, $V_1(x)=V_2(x)=0$ and $1+\frac{\alpha}{N}<p,q<1+\frac{\alpha+2}{N}$. Then there exists a solution $(u,v,\lambda_1,\lambda_2)$ of \eqref{system}, where $u$ and $v$ are positive and radial.
\end{theorem}

If $1+\frac{\alpha+2}{N}<p,~q<\frac{N+\alpha}{N-2}$, then the Choquard terms are $L^2$-supercritical and $E(u,v)$ is unbounded from below on $S_{\xi}\times S_{\eta}$. Therefore, \eqref{problem} doesn't work. We adopt mountain pass structure to obtain the solutions of \eqref{system}. Denote that $\delta_{p}=\frac{N(p-1)-\alpha}{2p}$ and
\begin{align}\label{h}
h(x)=\frac{1}{2}x-\frac{C}{2p}x^{p\delta_{p}}.
\end{align}
\begin{theorem}\label{t2}
Assume that $1+\frac{\alpha+2}{N}<p=q<\frac{N+\alpha}{N-2}$ and $V_1(x)=V_2(x)=0$. Let $\beta(x)$ satisfy
\begin{itemize}
  \item [(i)] $\beta(x)=\beta(|x|)>0$,
  \item [(ii)] $\beta(x),~x\cdot\nabla \beta(x)$ are bounded, moreover, $\|\beta(x)\|_{\infty}< \frac{\max\{h(x)\}}{2\xi\eta}$,
  \item [(iii)] $2\beta(x)+\frac{x\cdot\nabla \beta(x)}{\delta_{p}}\ge0$.
\end{itemize}
Then there exists a solution $(u,v,\lambda_1,\lambda_2)$ of \eqref{system}, where $u$ and $v$ are positive and radial for some $\lambda_1,\lambda_2>0$.
\end{theorem}
\begin{remark}
In order to simplicity some notations, we only consider $p=q$. But for $p\neq q$, we just do some technical generalization.
\end{remark}

Next, we consider the linearly coupled Choquard system \eqref{system} with different potentials.

For the $L^2$-subcritical case, i.e, $1+\frac{\alpha}{N}<p,~q<1+\frac{\alpha+2}{N}$, we give two different potentials:
\begin{itemize}
  \item \textbf{(V1)}:$V(x)\in C(\mathbb{R}^N,\mathbb{R})$, $\lim\limits_{|x|\rightarrow\infty}V(x)=0$ and $V(x)<0$ for all $x\in\mathbb{R}^N$;
  \item \textbf{(V2)}:$V(x)\in C(\mathbb{R}^N,\mathbb{R})$ and $\lim\limits_{|x|\rightarrow\infty}V(x)=\infty$.
\end{itemize}

\begin{theorem}\label{tv1}
Assume that $V_1(x)$ and $V_2(x)$  satisfy \textbf{(V1)}, $\max\{1+\frac{\alpha}{N},2\}<p,q<1+\frac{\alpha+2}{N}$ and $\beta>0$. Then there exists a ground state solution for $(u,v,\lambda_1,\lambda_2)\in H^1(\mathbb{R}^N)\times H^1(\mathbb{R}^N)\times\mathbb{R}\times\mathbb{R}$ of the system  \eqref{system}, where $u$ and $v$ are positive.
\end{theorem}
We define
\begin{align*}
\tilde{H}_{i}:=\{u\in H^1(\mathbb{R}^N):\int_{\mathbb{R}^N}V_{i}(x)|u|^2dx<\infty\}
\end{align*}
for $i=1,2$.
\begin{theorem}\label{tv2}
Assume that $V_1(x)$ and $V_2(x)$  satisfy \textbf{(V2)}, $\max\{1+\frac{\alpha}{N},2\}<p,q<1+\frac{\alpha+2}{N}$ and $\beta>0$. Then there exists a ground state solution $(u,v,\lambda_1,\lambda_2)\in \tilde{H}_{1}\times\tilde{H}_{2}\times\mathbb{R}\times\mathbb{R}$ for the system  \eqref{system}, where $u$ and $v$ are positive.
\end{theorem}
\begin{theorem}\label{tv1v2}
Assume that $V_1(x)$ satisfies \textbf{(V1)} and $V_2(x)$ satisfies  \textbf{(V2)}, $\max\{1+\frac{\alpha}{N},2\}<p,q<1+\frac{\alpha+2}{N}$ and $\beta>0$. Then  the system \eqref{system} has a solution $(u,v,\lambda_1,\lambda_2)\in H^1(\mathbb{R}^N)\times\tilde{H}_{2}\times\mathbb{R}\times\mathbb{R}$, where $u$ and $v$ are positive.
\end{theorem}
\begin{remark}
In the above theorem, if $V_1(x)$ satisfies \textbf{(V2)} and $V_2(x)$ satisfies  \textbf{(V1)}, then the theorem still holds.
\end{remark}

In this paper, the essential difficulty is to obtain the $L^2$-strong convergence of critical sequence due to the constraint \eqref{constraint}.

For $1+\frac{\alpha}{N}<p,q<1+\frac{\alpha+2}{N}$ and $V_1(x)=V_2(x)=0$, the key step is proved the subadditivity of $e(\xi,\eta)$(see Lemma \ref{lsubadd}), combining with Br\'ezis-Lieb lemma, it can lead to the $L^2$-strong convergence of minimizing sequence of $e(\xi,\eta)$.

For $1+\frac{\alpha+2}{N}<p=q<\frac{N+\alpha}{N-2}$ and $V_1(x)=V_2(x)=0$, applying  the  properties of the functional $E(u,v)$, we construct the mountain pass structure and obtain the bounded Palais-Smale sequence satisfying Pohozev identity in limit case. In addition, the sign of Lagrange multipliers plays an important role in
$L^2$-strong convergence of Palais-Smale sequence. Hence the difficulty can be turned into showing the sign of Lagrange multipliers. By $E'(u,v)\to 0$ and Pohozev identity in limit case, we can verify the sign of one of Lagrange multipliers. But for the other, by the fact the linearly coupled system \eqref{system} has no semi-trivial solutions, we can obtain the sign of the other Lagrange multiplier. Therefore, the $L^2$-strong convergence of Palais-Smale sequence is established.

For $1+\frac{\alpha}{N}<p,q<1+\frac{\alpha+2}{N}$ and $V_1(x)\neq 0, V_2(x)\neq0$, if $V_1(x)$ and $V_2(x)$  are trapping potentials, then
the $L^2$-strong convergence of Palais-Smale sequence can be given by the compact embedding lemma, which can be seen in \cite{R93}. If $V_1(x)$ or  $V_2(x)$ is not trapping potential, then the key step is to establish the split lemma for $e_{V}(\xi,\eta)$, see Lemma \ref{sub-lsplit} and Lemma \ref{lssplit}, combining with Br\'ezis-Lieb lemma, it can lead to the $L^2$-strong convergence of minimizing sequence of $e_{V}(\xi,\eta)$.

In our paper, there is a unsolved problem, which is the existence of normalized solutions for the linearly coupled with potentials for  $L^2$-supercritical case. Indeed, by comparing the energy, the proper linking geometry can be constructed and the bounded Palais-Smale sequence $\{(u_n,v_n)\}$ can be obtained. However, we can not verity the strong convergence of the Palais-Smale sequence $\{(u_n,v_n)\}$.  Similar to Lemma 3.1 in \cite{BC87}, there exists a $(u,v)$ such that $(u_n,v_n)\rightharpoonup(u,v)$ weakly in $H^1(\mathbb{R}^N)\times H^1(\mathbb{R}^N)$. Moreover, there exist an integer $k\ge0$, $k$ sequences of points $y^{j}_{n}(1\le j\le k)$ with $|y^{j}_{n}|\rightarrow\infty$ as $n\rightarrow\infty$ and nontrivial solutions $(w_{1j},w_{2j})(1\le j\le k)$ for the system \eqref{system} with $V_1(x)=V_2(x)=0$ such that
\begin{align*}
\|u_n-u-\sum_{i=1}^{k}w_{1i}(\cdot-y_{n}^{i})\|_{H^1(\mathbb{R}^N)}\rightarrow0,~\|v_n-v-\sum_{i=1}^{k}w_{2i}(\cdot-y_{n}^{i})\|_{H^1(\mathbb{R}^N)}\rightarrow0.
\end{align*}
By the upper bound of the energy, we can exclude the case $k\ge2$. Since the uniqueness of the solutions for the system \eqref{system} with $V_1(x)=V_2(x)=0$ is not clear, we can not exclude the case $k=1$ by comparing the energy.

The paper is organized as follows. In the Section 2, we give some preliminaries.  In the Section 3, we prove Theorem \ref{t1}. In the Section 4, we prove Theorem \ref{t2}. In the Section 5, we prove Theorem \ref{tv1}, Theorem \ref{tv2} and Theorem \ref{tv1v2}.

\section{Preliminaries}
In this section, we give some  preliminaries which will be used in later.

For simplicity, denote that
$$
B(u,p)=\int_{\mathbb{R}^N}(I_{\alpha}\star|u|^p)|u|^{p}dx=\int_{\mathbb{R}^N}\int_{\mathbb{R}^N}\frac{|u(x)|^{p}|u(y)|^{p}}{|x-y|^{N-\alpha}}dydx.
$$

First, we list some inequalities such as Gagliardo-Nirenberg inequality and Hardy-Littewood-Sobolev inequality \cite{LL01}.
\begin{lemma}
\begin{itemize}
  \item [(1)] Let $p\in[2,2^{*})$. If $N\ge3$ and $p\ge2$ if $N=1,2$. Then,
\begin{align}\label{Gagliardo-Nirenberg inequality}
\|f\|_{p}\le C_{N,p}\|\nabla f\|_2^{\gamma_p}\|f\|_2^{1-\gamma_p}
\end{align}
 with $\gamma_p=N(\frac12-\frac1p)$.
  \item [(2)]Let $t,r>1$ and $\alpha\in(0,N)$ with $\frac{1}{t}+\frac{1}{r}=1+\frac{\alpha}{N}$, $f\in L^{t}(\mathbb{R}^N)$ and $h\in L^{r}(\mathbb{R}^N)$. Then there exists a sharp constant $C(N,\alpha, t,r)>0$ such that
      \begin{align}\label{Hardy-Littewood-Sobolev inequality}
      \bigg|\int_{\mathbb{R}^N}\int_{\mathbb{R}^N}\frac{|f(x)||h(y)|}{|x-y|^{N-\alpha}}dxdy\bigg|\le C(N,\alpha, t,r)\|f\|_{t}\|h\|_{r}.
      \end{align}
      Moreover, if $t=r=\frac{2N}{N+\alpha}$, then
      \begin{align*}
      C(N,\alpha):=C(N,\alpha,t,r)=\pi^{\frac{N-\alpha}{2}}\frac{\Gamma(\frac{\alpha}{2})}{\Gamma(\frac{N+\alpha}{2})}\bigg\{\frac{\Gamma(\frac{N}{2})}{\Gamma(N)}\bigg\}^{-\frac{\alpha}{N}}.
      \end{align*}
\end{itemize}
\end{lemma}

\begin{lemma}[Weak Young inequality \cite{LL01}] Let $N\in\mathbb{N},~\alpha\in(0,N),~p,r>1$ and $\frac{1}{p}=\frac{\alpha}{N}+\frac{1}{r}$. If $v\in L^{p}(\mathbb{R}^N)$, then $I_{\alpha}\star v\in L^{r}(\mathbb{R}^N)$
 and
$$
 \bigg(\int_{\mathbb{R}^N}|I_{\alpha}\star v|^{r}\bigg)^{\frac{1}{r}}\le C(N,\alpha,p)\bigg(\int_{\mathbb{R}^N}|v|^{p}\bigg)^{\frac{1}{p}}.
$$
In particular, we can set $p=\frac{N}{\alpha}$ and $r=+\infty$.
\end{lemma}

In \eqref{Hardy-Littewood-Sobolev inequality}, take $t=r=\frac{2N}{N+\alpha}$ and $f=h=|u|^{p}$, together with \eqref{Gagliardo-Nirenberg inequality}, we have
\begin{align}\label{bu}
B(u,p)\le C(N,\alpha)\|u\|_{\frac{2Np}{N+\alpha}}^{2p}\le C(N,s,\alpha,p)\|\nabla u\|_2^{2p\delta_{p}}\|u\|_2^{2p(1-\delta_{p})},
\end{align}
where $\delta_{p}=\frac{N(p-1)-\alpha}{2p}$.  Furthermore, it is easy to observe that $B(u,p)$ is well defined for $p\in(1+\frac{\alpha}{N},\frac{N+\alpha}{N-2})$.

\begin{lemma}\label{lbw}\cite{Lx20}
Assume that $N>2s,~\alpha\in(0,N)$ and $r\in[1+\frac{\alpha}{N},\frac{N+\alpha}{N-2s}]$. Suppose that $\{u_n\}_{n=1}^{\infty}\subset H^{s}(\mathbb{R}^N)$ satisfy $u_n\rightharpoonup u$  weakly in $H^{s}(\mathbb{R}^N)$ as $n\rightarrow\infty$. Then
\begin{align*}
\int_{\mathbb{R}^N}(I_{\alpha}\star|u_n|^p)|u_n|^{p-2}u_n\varphi\rightarrow\int_{\mathbb{R}^N}(I_{\alpha}\star|u|^p)|u|^{p-2}u\varphi\text{~weakly in }H^{-s}(\mathbb{R}^N)\text{~as~}n\rightarrow\infty,
\end{align*}
for any $\varphi\in H^{s}(\mathbb{R}^N)$.
\end{lemma}

Next we give Br\'ezis-Lieb lemma for the nonlocal term of the functional, which can be seen in \cite{MV13}.
\begin{lemma}\label{lbl}
Let $N\in\mathbb{N},~\alpha\in(0,N),~p\in[1,\frac{2N}{N+\alpha})$ and $\{u_n\}$ be a bounded sequence in $L^{\frac{2Np}{N+\alpha}}$. If $u_n\to u$ almost everywhere on $\mathbb{R}^N$ as $n\to\infty$, then
$$
\lim_{n\to\infty}\int_{\mathbb{R}^N}(I_{\alpha}\star|u_n|^{p})|u_n|^{p}-\int_{\mathbb{R}^N}(I_{\alpha}\star|u_n-u|^{p})|u_n-u|^{p}
=\int_{\mathbb{R}^N}(I_{\alpha}\star|u|^{p})|u|^{p}.$$
\end{lemma}
\begin{lemma}\cite{LL01}\label{lra}
Let $f,g$ and $h$ be three Lebesgue measurable non-negative functions on $\mathbb{R}^N$. Then, with
$$
\Psi(f,g,h)=\int_{\mathbb{R}^N}\int_{\mathbb{R}^N}f(x)g(x-y)h(y)dxdy,
$$
we have
$$
\Psi(f,g,h)\le\Psi(f^{*},g^{*},h^{*}),
$$
where $f^{*}$ is the Schwartz rearrangement of $f$.
\end{lemma}

Next we consider the single equation
\begin{align}\label{scalar problem}
\left\{\begin{aligned}
&-\Delta u+\lambda u=\mu(I_{\alpha}\star|u|^p)|u|^{p-2}u~\quad x\in \mathbb{R}^N,\\
&\int_{\mathbb{R}^N}u^2dx=c^2.
\end{aligned}
\right.
\end{align}
The corresponding functional with \eqref{scalar problem} is
\begin{align*}
F_{\mu}(u)=\frac{1}{2}\|\nabla u\|_2^2-\frac{\mu}{2p}B(u,p).
\end{align*}
Denote that
\begin{align}\label{mcnc}
\begin{split}
&m(c,\mu):=\inf_{u\in S_{c}}F_{\mu}(u)\text{ for }1+\frac{\alpha}{N}<p<1+\frac{\alpha+2}{N},\\
&n(c,\mu):=\inf_{\mathcal{N}_{c}}F_{\mu}(u) \text{ for }1+\frac{\alpha+2}{N}<p<\frac{N+\alpha}{N-2},
\end{split}
\end{align}
where
\begin{align*}
G(u)=\|\nabla u\|_2^2-\mu\delta_{p}B(u,p)\text{~and~}\mathcal{N}_{c}=\{u\in S_{c}:G(u)=0\}.
\end{align*}
We collect some properties about the scalar equation which can be seen in \cite{CL82,LL20,LY14}.
\begin{lemma}\label{yt1}
\begin{itemize}
\item[(i)] Assume that $1+\frac{\alpha}{N}<p<1+\frac{\alpha+2}{N}$. Then
\begin{enumerate}
  \item $m(c,\mu)<0$ for all $c>0$.
  \item $m(c,\mu)<m(c-\alpha,\mu)+m(\alpha,\mu)$ for any $\alpha\in(0,c)$.
\end{enumerate}
\item[(ii)]Assume that $1+\frac{\alpha+2}{N}<p<\frac{N+\alpha}{N-2}$. Then
  $n(c,\mu)$ is strictly decreasing with respect to $c$.
\end{itemize}
\end{lemma}

\begin{lemma}\cite{WY17}
Suppose that $N=3,4,5$, if $u$ is a positive solution of
\begin{align*}
-\Delta u+u=\int_{\mathbb{R}^N}\frac{|u(y)|^2}{|x-y|^{N-2}}dyu,\quad u\in\mathbb{R}^N,
\end{align*}
then $u$ is unique up to translations.
\end{lemma}

\section{$L^2$-subcritical case }
In this section, we prove the existence of solutions for $L^2$-subcritical case without potential. In order to show Theorem \ref{t1}, we need some lemmata.

\begin{lemma}\label{lbfb}
$E(u,v)$ is coercive and bounded from below on $S_{\xi}\times S_{\eta}$.
\end{lemma}
\begin{proof}
By \eqref{bu}, it holds that
\begin{align*}
&B(u,p)\le C(N,\alpha,p)\|\nabla u\|_2^{2p\delta_p}\xi^{2p(1-\delta_{p})},\\
&B(v,q)\le C(N,\alpha,q)\|\nabla v\|_2^{2q\delta_q}\eta^{2q(1-\delta_{q})},
\end{align*}
where $\delta_p=\frac{N(p-1)-\alpha}{2p}$.  Thus,
\begin{align*}
E(u,v)\ge& \frac12(\|\nabla u\|_2^2+\|\nabla v\|_2^2)-\frac{\mu_1C(N,\alpha,p)\xi^{2p(1-\delta_{p})}}{2p}\|\nabla u\|_2^{2p\delta_p}\\
&-\frac{\mu_2C(N,\alpha,q)\eta^{2q(1-\delta_{q})}}{2q}\|\nabla v\|_2^{2q\delta_q}-\beta \xi\eta.
\end{align*}
Since $1+\frac{\alpha}{N}<p,q<1+\frac{\alpha+2}{N}$, $2p\delta_p<2$ and $2q\delta_q<2$. Hence we complete the proof.
\end{proof}
From Lemma \ref{lbfb}, there exists a minimizing sequence $\{(u_n,v_n)\}$ for $E|_{S_{\xi}\times S_{\eta}}$. Due to $\beta>0$ and $\|\nabla |u|\|_2^2\le\|\nabla u\|_2^2$, we can assume that $u_n,v_n\ge0$.

In the following, we give the subadditivity of $e(\xi,\eta)$.
\begin{lemma}\label{lsubadd}
There holds that
$e(\xi,\eta)\le e(\xi_1,\eta_1)+e(\xi-\xi_1,\eta-\eta_1)$.
\end{lemma}
\begin{proof}
Denote that $\gamma=(\xi,\eta),~\gamma_1=(\xi_1,\eta_1)\in[0,\xi]\times[0,\eta]$ and $S(\gamma):=S_{\xi}\times S_{\eta}$, one need to show that
\begin{align*}
e(\gamma)\le e(\gamma_1)+e(\gamma-\gamma_1).
\end{align*}
For any $\varepsilon>0$, we may find $\varphi_{\varepsilon},\psi_{\varepsilon}\in C_c^{\infty}(\mathbb{R}^N)\times C_c^{\infty}(\mathbb{R}^N)$ such that
\begin{align*}
&supp\varphi_{\varepsilon}\cap supp\psi_{\varepsilon}=\emptyset,~ \varphi_{\varepsilon}\in S(\gamma_1),~E(\varphi_{\varepsilon})\le e(\gamma_1)+\varepsilon,\\
&\psi_{\varepsilon}\in S(\gamma-\gamma_1),~E(\psi_{\varepsilon})\le e(\gamma-\gamma_1)+\varepsilon.
\end{align*}
Set $u_{\varepsilon}=\varphi_{\varepsilon}+\psi_{\varepsilon}$. Since $supp\varphi_{\varepsilon}\cap supp\psi_{\varepsilon}=\emptyset$, we can observe that $u_{\varepsilon}\in S(\gamma)$. Moreover,
\begin{align*}
e(\gamma)\le E(u_{\varepsilon})\le E(\varphi_{\varepsilon})+E(\psi_{\varepsilon})\le  e(\gamma_1)+ e(\gamma-\gamma_1)+2\varepsilon.
\end{align*}
Since $\varepsilon$ is arbitrary, we can establish the lemma.
\end{proof}
\begin{remark}\label{r31}
If $e(\gamma)$ is achieved, then $e(\gamma)<e(\gamma_1)+ e(\gamma-\gamma_1)$.
\end{remark}

Let $u^{*}$  be the Schwartz rearrangement of $u$.  Then $\|\nabla u^{*}\|_2^2\le\|\nabla u\|_2^2$.  Denote that $S_{c}^{r}:=S_{c}\cap H_{r}^{1}(\mathbb{R}^N)$.

\begin{lemma}
There holds that
\begin{align*}
\inf_{S_{\xi}^{r}\times S_{\eta}^{r}}E(u,v)=\inf_{S_{\xi}\times S_{\eta}}E(u,v).
\end{align*}
\end{lemma}
\begin{proof}
On the one hand, since $S_{\xi}^{r}\times S_{\eta}^{r}\subset S_{\xi}\times S_{\eta}$, it can be derived that
\begin{align*}
\inf_{S_{\xi}^{r}\times S_{\eta}^{r}}E(u,v)\ge\inf_{S_{\xi}\times S_{\eta}}E(u,v).
\end{align*}
On the other hand, for any $(u,v)\in S_{\xi}\times S_{\eta}$, let $(u^{*},v^{*})$ be the Schwartz rearrangement of $(u,v)$. By Lemma \ref{lra}, it holds that $B(u^{*},p)\ge B(u,p)$.  Therefore,
\begin{align*}
E(u^{*},v^{*})
=&\frac{1}{2}(\|\nabla u^{*}\|_{2}^{2}+\|\nabla v^{*}\|_2^2)
-\frac{\mu_1}{2p}B(u^{*},p)-\frac{\mu_2}{2q}B(v^{*},q) -\beta\int_{\mathbb{R}^N}u^{*}v^{*}dx\\
\le&\frac{1}{2}(\|\nabla u\|_{2}^{2}+\|\nabla v\|_2^2)
-\frac{\mu_1}{2p}B(u,p)-\frac{\mu_2}{2q}B(v,q) -\beta\int_{\mathbb{R}^N}uvdx,
\end{align*}
hence the lemma holds.
\end{proof}
\begin{remark}\label{r1}
From the above lemma, we can suppose that $(u_n,v_n)\in H_{r}^{1}(\mathbb{R}^N)\times H_{r}^{1}(\mathbb{R}^N)$. From Lemma \ref{lbfb}, it can be deduced that $(u_n,v_n)$ is bounded in $H_{r}^{1}(\mathbb{R}^N)\times H_{r}^{1}(\mathbb{R}^N)$. Therefore, there exists $(u_0,v_0)\in H_{r}^{1}(\mathbb{R}^N)\times H_{r}^{1}(\mathbb{R}^N)$ such that $(u_n,v_n)\rightharpoonup(u_0,v_0)$ weakly in $H_{r}^{1}(\mathbb{R}^N)\times H_{r}^{1}(\mathbb{R}^N)$ and strongly in $L^{t}(\mathbb{R}^N)\times L^{t}(\mathbb{R}^N)$ for $t\in(2,2_{s}^{*})$. By \eqref{bu}, it follows that $B(u_n,p)\rightarrow B(u_0,p)$ and $B(v_n,q)\rightarrow B(v_0,q)$.
\end{remark}
In order to obtain the critical point of $E(u,v)$ on $S_{\xi}\times S_{\eta}$, we need to show the strong convergence of minimizing sequence in $L^2(\mathbb{R}^N)\times L^2(\mathbb{R}^N)$.

\begin{lemma}\label{ll2c}
There holds that
\begin{align*}
\|u_0\|_2^2=\xi^2,~\|v_0\|_2^2=\eta^2.
\end{align*}
\end{lemma}
\begin{proof}
From Remark \ref{r1}, $(u_n,v_n)\rightharpoonup(u_0,v_0)$ weakly in $H_{r}^{1}(\mathbb{R}^N)\times H_{r}^{1}(\mathbb{R}^N)$. Consequently,
\begin{align*}
\|u_0\|_2^2\le\liminf_{n\rightarrow\infty}\|u_n\|_2^2=\xi^2,~\|v_0\|_2^2\le\liminf_{n\rightarrow\infty}\|v_n\|_2^2=\eta^2.
\end{align*}
Next we show that $"="$ holds by contradiction. Set $\|u_0\|_2=\xi_1~\|v_0\|_2=\eta_1$. There exist three cases:
\begin{align*}
\left\{
\begin{array}{ll}
\textbf{Case } 1: & \hbox{$\xi_1<\xi,~\eta_1<\eta$;} \\
\textbf{Case } 2: & \hbox{$\xi_1=\xi,~\eta_1<\eta$;} \\
\textbf{Case }3: & \hbox{$\xi_1<\xi,~\eta_1=\eta$.}
\end{array}
\right.
\end{align*}
Set $\hat{u}_n=u_n-u_0$ and $\hat{v}_n=v_n-v_0$, denote that $\|\hat{u}_n\|_2^2:=\xi_2^2=\xi^2-\xi_1^2,~\|\hat{v}_n\|_2^2:=\eta_2^2=\eta^2-\eta_1^2$.

For \textbf{Case} 1, we have $\xi_2>0$ and $\eta_2>0$. By Br\'ezis-Lieb lemma in \cite{BL83}, it holds
\begin{align}\label{e01}
\begin{split}
e(\xi,\eta)+o_n(1)=E(u_n,v_n)&= E(u_0,v_0)+E(\hat{u}_n,\hat{v}_n)+o_n(1)\\
&\ge e(\xi_1,\eta_1)+E(\hat{u}_n,\hat{v}_n).
\end{split}
\end{align}
From Remark \ref{r1}, $B(\hat{u}_n,\hat{v}_n)=0$. As a consequence,
\begin{align}\label{ehat}
E(\hat{u}_n,\hat{v}_n)=\frac{1}{2}(\|\nabla\hat{u}_n\|_2^2+\|\nabla\hat{v}_n\|_2^2)-\beta\int_{\mathbb{R}^N}\hat{u}_n\hat{v}_ndx+o(1)\ge-\beta\xi_2\eta_2+o(1).
\end{align}
Without loss of generality, we may assume that $\xi_2^2\ge\eta_2^2$. For any $u\in S_{\xi_2}$,  it yields that $(u,\frac{\eta_2}{\xi_2}u)\in S_{\xi_2}\times S_{\eta_2}$. Therefore,
\begin{align*}
E(u,\frac{\eta_2}{\xi_2}u)
=&\frac{1}{2}\bigg(1+\frac{\eta_2^2}{\xi_2^2}\bigg)\|\nabla u\|^2
-\frac{\mu_1}{2p}B(u,p)-\frac{\mu_2}{2q}\frac{\eta_2^{2q}}{\xi_2^{2q}}B(u,q)-\beta\xi_2\eta_2\\
\le&2\bigg(\frac{1}{2}\|\nabla u\|^2-\frac{\mu_1}{2}\frac{1}{2p}B(u,p)\bigg)-\beta\xi_2\eta_2\\
=&2F_{\frac{\mu_1}{2}}(u)-\beta\xi_2\eta_2.
\end{align*}
Applying  Lemma
\ref{yt1}, let $u$ satisfy $m(\xi_2,\frac{\mu_1}{2})=F_{\frac{\mu_1}{2}}(u)<0$. Then
\begin{align}\label{e2}
e(\xi_2,\eta_2)\le E(u,\frac{\eta_2}{\xi_2}u)\le 2F_{\frac{\mu_1}{2}}(u)-\beta\xi_2\eta_2.
\end{align}
By \eqref{e01}-\eqref{e2} and Lemma \ref{lsubadd},  we can derive that
\begin{align*}
e(\xi,\eta)&\ge e(\xi_1,\eta_1)+e(\xi_2,\eta_2)-2F_{\frac{\mu_1}{2}}(u)\\
&> e(\xi_1,\eta_1)+e(\xi_2,\eta_2)\\
&\ge e(\xi,\eta),
\end{align*}
which is a contradiction. Hence \textbf{Case} 1  does not hold.

Next we show \textbf{Case} 2 does not hold by contradiction. For \textbf{Case} 2, we have $\xi_2=0,~\eta_2>0$. Similar to \textbf{Case} 1, it can be derived that
\begin{align*}
e(\xi,\eta)= e(\xi,\eta_1)+E(\hat{u}_n,\hat{v}_n)+o_n(1),
\end{align*}
by $B(u_n,p)\rightarrow B(u_0,p),B(v_n,q)\rightarrow B(v_0,q)$ and $u_n,v_n\ge0$, we have
\begin{align*}
E(\hat{u}_n,\hat{v}_n)
=\frac{1}{2}(\|\nabla\hat{u}_n\|_2^2+\|\nabla\hat{v}_n\|_2^2)-\beta\int_{\mathbb{R}^N}\hat{u}_n\hat{u}_ndx+o(1)\ge o(1).
\end{align*}
where $0\le\beta\int_{\mathbb{R}^N}\hat{u}_n\hat{v}_ndx\le\beta\|\hat{u}_n\|_2\|\hat{v}_n\|_2=0$.
From $e(0,\eta_2)=m(\eta_2,\mu)=F_{\mu}(u)+o(1)<0$, we obtain  that
\begin{align*}
e(\xi,\eta)&\ge e(\xi,\eta_1)+E(\hat{u}_n,\hat{v}_n)\ge e(\xi,\eta_1)+o(1)\\
&=e(\xi,\eta_1)+e(0,\eta_2)-F_{\mu}(u)+o(1)>e(\xi,\eta),
\end{align*}
which is a contradiction. Hence \textbf{Case} 2 does not hold. Moreover, similar to the proof of \textbf{Case} 2, \textbf{Case} 3 does not hold. Hence the lemma holds.
\end{proof}

\begin{proof}[\textbf{Proof of Theorem \ref{t1}}]
From Lemma \ref{ll2c}, $(u_n,v_n)\rightarrow(u_0,v_0)$ strongly in $L^2(\mathbb{R}^N)\times L^2(\mathbb{R}^N)$. Applying $(u_n,v_n)$ is a minimizing sequence, there holds that
\begin{align*}
e(\xi,\eta)\le E(u_0,v_0)\le\liminf_{n\rightarrow\infty}E(u_n,v_n)=e(\xi,\eta),
\end{align*}
which implies that $(u_0,v_0)$ is a minimizer. Moreover, $(u_0,v_0)$ satisfies
\begin{align*}
\left\{\begin{aligned}
&-\Delta u_0+\lambda_1 u_0=\mu_1(I_{\alpha}\star|u_0|^p)|u_0|^{p-2}u+\beta v_0,\\
&-\Delta v_0+\lambda_2 v_0=\mu_2(I_{\alpha}\star|v_0|^q)|v_0|^{q-2}u+\beta u_0,
\end{aligned}
\right.
\end{align*}
with $\|u_0\|_2^2=\xi^2$ and $\|v_0\|_2^2=\eta^2$. Finally by maximum principle and $u_0,v_0\ge0$, there holds that $u_0,v_0>0$. Hence the theorem is established.
\end{proof}

\section{$L^2$-supercritical case}
In this section, we mainly consider $1+\frac{\alpha+2}{N}<p<\frac{\alpha+N}{N-2}$. We work in radial setting. For any $(u,v)\in H^{1}_{r}(\mathbb{R}^N)\times H^{1}_{r}(\mathbb{R}^N)$, define the map
\begin{align*}
s\star(u,v)=(s\star u,s\star v):=(e^{\frac{Ns}{2}}u(e^{sx}),e^{\frac{Ns}{2}}v(e^{sx})).
\end{align*}
It is obvious to observe that $\|s\star u\|_2^2=\|u\|_2^2$ and $\|s\star v\|_2^2=\|v\|_2^2$. Denote that $S_{c}^{r}=S(c)\cap H^{1}_{r}(\mathbb{R}^N)$.
\begin{lemma}\label{sinfty}
Let $(u,v)\in S^{r}_{\xi}\times S^{r}_{\eta}$. Then
\begin{align*}
\lim_{s\rightarrow-\infty}\|\nabla(s\star u)\|_2^2+\|\nabla(s\star v)\|_2^2=0^{+},~\lim_{s\rightarrow\infty}\|\nabla(s\star u)\|_2^2+\|\nabla(s\star v)\|_2^2=\infty,
\end{align*}
and
\begin{align*}
\lim_{s\rightarrow-\infty}E(s\star(u,v))=-\lim_{s\rightarrow-\infty}\int_{\mathbb{R}^N}\beta(e^{-s}x)uvdx,~\lim_{s\rightarrow\infty}E(s\star(u,v))=-\infty.
\end{align*}
\end{lemma}
\begin{proof}
By direct computation, the lemma can be established. Hence we omit the details.
\end{proof}

\begin{lemma}\label{lh}
Let $h(x)$ be defined as in \eqref{h}. Then there holds
\begin{itemize}
  \item[(i)] there exists a $x_0>0$ such that $h(0)=h(x_0)=0$;
  \item[(ii)]there exists a $x_1>0$ such that $h(x_1)=\max\limits_{x\ge0}\{h(x)\}$.
\end{itemize}
\end{lemma}
\begin{proof}
Direct computation shows that
\begin{align*}
h'(x)=\frac{1}{2}-\frac{Cp\delta_{p}}{2p}x^{p\delta_{p}-1},
\end{align*}
due to $p\delta_{p}-1>0$, it is easily to obtain that $h'(x)=0$ has a unique solution, denote by $x_1$. Hence $h(x)$ is increasing in $(0,x_1)$ and decreasing in $(x_1,+\infty)$. Since $h(0)=0,~h(x_1)=\max\{h(x)\}>0$,  there exists a $x_0>0$ such that $h(0)=h(x_0)=0$. Therefore, we can establish the lemma.
\end{proof}

By \eqref{bu}, it follows that
\begin{align}\label{estimate u^p}
\begin{split}
\mu_1B(u,p)+\mu_2B(v,p)\le &C(N,\alpha,p)(\mu_1\|\nabla u\|_2^{2p\delta_{p}}\xi^{2p(1-\delta_{p})}+\mu_2\|\nabla v\|_2^{2p\delta_{p}}\eta^{2p(1-\delta_{p})})\\
=:&
C_{\xi,\eta}(N,\alpha,p,\mu_1,\mu_2)(\|\nabla u\|_2^{2p\delta_{p}}+\|\nabla v\|_2^{2p\delta_{p}})\\
\le&C_{\xi,\eta}(N,\alpha,p,\mu_1,\mu_2)(\|\nabla u\|_2^2+\|\nabla v\|_2^2)^{p\delta_{p}}.
\end{split}
\end{align}
The next lemma enlightens the mountain pass structure of the problem.
\begin{lemma}\label{lmps}
There exist a constant $K_1>0$ sufficient small, a constant $K_2$ satisfying $K_2>K_1$  and $\|\beta(x)\|_{\infty}< \frac{\max\{h(x)\}}{2\xi\eta}$. Define the sets
\begin{align*}
&\Omega=\{(u,v)\in S^{r}_{\xi}\times S^{r}_{\eta},\|\nabla u\|_2^2+\|\nabla v\|_2^2\le K_1\},\\
&\Pi=\{(u,v)\in S^{r}_{\xi}\times S^{r}_{\eta},\|\nabla u\|_2^2+\|\nabla v\|_2^2= K_2\}.
\end{align*}
There holds
\begin{itemize}
  \item [(i)] for any $(u,v)\in\Omega$, there exists $C_1>0$ such that $E(u,v)>-C_1$;
  \item [(ii)] $\sup\limits_{\Omega}E(u,v)<\inf\limits_{\Pi}E(u,v)$;
  \item[(iii)]$\inf\limits_{\Pi}E(u,v)>0$.
\end{itemize}
\end{lemma}
\begin{proof}
(i) For any $(u,v)\in\Omega$, by \eqref{estimate u^p}, it follows that
\begin{align*}
E(u,v)=&\frac{1}{2}(\|\nabla u\|_2^2+\|\nabla v\|_2^2)-\frac{1}{2p}(\mu_1B(u,p)+\mu_2B(v,p))-\int_{\mathbb{R}^N}\beta(x)uvdx\\
\ge&\frac{1}{2}(\|\nabla u\|_2^2+\|\nabla v\|_2^2)
-\frac{C_{\xi,\eta}(N,\alpha,p,\mu_1,\mu_2)}{2p}(\|\nabla u\|_2^2+\|\nabla v\|_2^2)^{p\delta_{p}}-\|\beta(x)\|_{\infty}\xi\eta.
\end{align*}
Taking $C=C_{\xi,\eta}(N,\alpha,p,\mu_1,\mu_2)$ and $x=\|\nabla u\|_2^2+\|\nabla v\|_2^2\le K_1<x_0$ in \eqref{h}, by Lemma \ref{lh}, we get that
$$
E(u,v)\ge-\|\beta(x)\|_{\infty}\xi\eta:=-C_1.
$$
Therefore (i) holds.

(ii) If $(u_1,v_1)\in\Pi$ and $(u_2,v_2)\in\Omega$, then by \eqref{estimate u^p},
\begin{align}\label{-}
\begin{split}
E(u_1,v_1)-E(u_2,v_2)
\ge\frac{K_2}{2}-\frac{C_{\xi,\eta}(N,\alpha,p,\mu_1,\mu_2)K_2^{p\delta_{p}}}{2p}-\frac{K_1}{2}-2\|\beta(x)\|_{\infty}\xi\eta.
\end{split}
\end{align}
From Lemma \ref{lh}, take $K_2=x_1$ and let $K_1$ be sufficiently small, if $\|\beta(x)\|_{\infty}<\frac{h(K_2)}{2\xi\eta}$, then
$$
E(u_1,v_1)-E(u_2,v_2)>0
$$
for any $(u_1,v_1)\in\Pi$ and $(u_2,v_2)\in\Omega$.
Therefore (ii) holds.

(iii) Similar to \eqref{-},  for any $(u_1,v_1)\in\Pi$, there holds
\begin{align*}
E(u_1,v_1)\ge&\frac{K_2}{2}-\frac{C_{\xi,\eta}(N,\alpha,p,\mu_1,\mu_2)K_2^{p\delta_{p}}}{2p}-\|\beta(x)\|_{\infty}\xi\eta>0
\end{align*}
for $\|\beta(x)\|_{\infty}<\frac{h(K_2)}{2\xi\eta}$.  Hence (iii) holds.
\end{proof}
Define
\begin{align*}
\Delta=\{(u,v)\in S^{r}_{\xi}\times S^{r}_{\eta},\|\nabla u\|_2^2+\|\nabla v\|_2^2\ge2K_2~\text{and}~E(u,v)\le 0\}.
\end{align*}
Notice that $\Delta\neq\emptyset$ by Lemma \ref{sinfty}. From Lemma \ref{sinfty} and Lemma \ref{lmps}, if $(\bar{u},\bar{v})\in\Omega$ and $(\hat{u},\hat{v})\in\Delta$, then there is a mountain path linking $(\bar{u},\bar{v})$ and $(\hat{u},\hat{v})$ and passing through $\Pi$. Define
\begin{align*}
\Gamma:=\{\gamma:=(\gamma_1(t),\gamma_2(t))\in C([0,1],S_{\xi}\times S_{\eta}):\gamma(0)=(\bar{u},\bar{v}),\gamma(1)=(\hat{u},\hat{v})\}.
\end{align*}
\begin{lemma}
There exists a Palais-Smale sequence $(\bar{u}_n,\bar{v}_n)$ for $E$ on $S^{r}_{\xi}\times S^{r}_{\eta}$ at the level
\begin{align*}
c:=\inf_{\gamma\in\Gamma}\max_{t\in[0,1]}E(\gamma(t))\ge\inf_{\Pi}E>0\ge \max\{E(\bar{u},\bar{v}),E((\hat{u},\hat{v}))\}
\end{align*}
satisfying additional condition
\begin{align}\label{add condition}
\begin{split}
&(\|\nabla \bar{u}_n\|_2^2+\nabla \bar{v}_n\|_2^2)
-\delta_{p}(\mu_1B(\bar{u}_n,p)+\mu_2B(\bar{v}_n,p))\\
+&\int_{\mathbb{R}^N}(x\cdot\nabla \beta(x))\bar{u}_n(x)\bar{v}_n(x)dx=o_n(1)
\end{split}
\end{align}
with $o_{n}(1)\rightarrow0$ as $n\rightarrow\infty$.
\end{lemma}
\begin{proof}
We consider the auxiliary functional $\bar{E}$:
\begin{align*}
&\bar{E}:\mathbb{R}\times S^{r}_{\xi}\times S^{r}_{\eta}\rightarrow\mathbb{R},\\
&\bar{E}(s,u,v)=E(s\star u,s\star v).
\end{align*}
Set
$$
\bar{\Gamma}=\{\bar{\gamma}(t)=(s(t),\gamma_1(t),\gamma_2(t))\in C([0,1],\mathbb{R}\times S^{r}_{\xi}\times S^{r}_{\eta}):\bar{\gamma}(0)=(0,\bar{u},\bar{v}),\bar{\gamma}(1)=(0,\hat{u},\hat{v})\},
$$
and define
\begin{align*}
\bar{c}:=\inf_{\bar{\gamma}\in\bar{\Gamma}}\sup_{t\in[0,1]}\bar{E}(\bar{\gamma}(t)).
\end{align*}

We next show $c=\bar{c}$. On the one hand, since $\Gamma\subset\bar{\Gamma}$, it is easy to find that $c\ge\bar{c}$. On the other hand, for any $\bar{\gamma}(t)=(s(t),\gamma_1(t),\gamma_2(t))\in\bar{\Gamma}$, it holds that
\begin{align*}
(s(t)\star\gamma_1(t),s(t)\star\gamma_2(t))\in\Gamma~\text{and}~\bar{E}(\bar{\gamma}(t))=E(s(t)\star\gamma_1(t),s(t)\star\gamma_2(t)),
\end{align*}
hence $c\le\bar{c}$. Therefore, $c=\bar{c}$.

Notice that
\begin{align*}
\bar{E}(\bar{\gamma}(t))=E(s(t)\star\gamma_1(t),s(t)\star\gamma_2(t))=\bar{E}(0,s(t)\star\gamma_1(t),s(t)\star\gamma_2(t)).
\end{align*}
Since $\beta(x)\ge0$, $\bar{E}(s,|u|,|v|)\le\bar{E}(s,u,v)$. As a consequence, we can choose the minimizing sequence $\bar{\gamma}_{n}=(s_n,\gamma_{1n},\gamma_{2n})$ for $\bar{c}$ satisfying
\begin{align*}
\gamma_{1n},~\gamma_{2n}\ge0 \text{ a.e. in }\mathbb{R}^N.
\end{align*}
Using Theorem 3.2 in \cite{G93}, there exists a Palais-Smale sequence $(s_n,u_n,v_n)\in\mathbb{R}\times S_{\xi}\times S_{\eta}$ for $\bar{E}$ at level $\bar{c}$ such that
\begin{itemize}
  \item [(1)]$\lim\limits_{n\rightarrow\infty}\bar{E}(s_n,u_n,v_n)=\bar{c}=c$,
  \item[(2)]$\lim\limits_{n\rightarrow\infty}|s_n|+dist((u_n,v_n),(\gamma_{1n},\gamma_{2n}))=0$,
  \item [(3)]$\lim\limits_{n\rightarrow\infty}\|\nabla_{\mathbb{R}\times S_{\xi}\times S_{\eta}}\bar{E}(s_n,u_n,v_n)\|=0$.
\end{itemize}
From (1), we have $E(s_n\star u_n,s_n\star v_n)=c,~s_n\rightarrow 0$.  From $\gamma_{1n},~\gamma_{2n}\ge0$ and (2), it follows that $u_n,v_n\ge0$. Therefore, $\{(\bar{u}_n,\bar{v}_n)\}=:\{(s_n\star u_n,s_n\star v_n)\}$ is a Palais-Smale sequence for $E$. Similarly, for any $(\varphi_n,\psi_n)\in H_1^r\times H_1^r$, setting $(\tilde{\varphi}_n,\tilde{\psi}_n)=((-s_n)\star \varphi_n,(-s_n)\star \psi_n)$, it is easily to find that $\int_{\mathbb{R}^N}\bar{u}_n\varphi_n=0$ and $\int_{\mathbb{R}^N}\bar{v}_n\psi_n=0$ is equivalent to $\int_{\mathbb{R}^N}u_n\tilde{\varphi}_n=0$ and $\int_{\mathbb{R}^N}v_n\tilde{\psi}_n=0$. Therefore,
\begin{align*}
DE(\bar{u}_n,\bar{v}_n)[(\varphi_n,\psi_n)]=D\bar{E}(s_n,u_n,v_n)[(0,\tilde{\varphi}_n,\tilde{\psi}_n)]+o(1)\|(\tilde{\varphi}_n,\tilde{\psi}_n)\|_{H^1(\mathbb{R}^N)\times H^1(\mathbb{R}^N)},
\end{align*}
due to $\|(\tilde{\varphi}_n,\tilde{\psi}_n)\|^2_{H^1(\mathbb{R}^N)\times H^1(\mathbb{R}^N)}\le 2\|(\varphi_n,\psi_n)\|^2_{H^1(\mathbb{R}^N)\times H^1(\mathbb{R}^N)}$ for $n$ large, it can be deduced that $\nabla_{S_{\xi}\times S_{\eta}}E(\bar{u}_n,\bar{v}_n)\rightarrow 0$. By (3), it implies that
\begin{align*}
D\bar{E}(s_n,u_n,v_n)[(1,0,0)]\rightarrow0, ~~s_n\rightarrow 0\text{~as~}n\rightarrow\infty.
\end{align*}
We can compute that
\begin{align*}
\partial_{s}(\int_{\mathbb{R}^N}\beta(x)e^{Ns}u(e^sx)v(e^sx)dx)&=\partial_{s}(\int_{\mathbb{R}^N}\beta(e^{-s}x)u(x)v(x)dx)\\
&=-\int_{\mathbb{R}^N}(e^{-s}x\cdot\nabla\beta(e^{-s}x))u(x)v(x)dx,
\end{align*}
and
\begin{align*}
\left\{\begin{aligned}
&\partial_{s}\bigg\{\frac{1}{2}(\|\nabla (s\star u)\|_2^2+\|\nabla (s\star v)\|_2^2)\bigg\}
=e^{2s}(\|\nabla u\|_2^2+\|\nabla  v\|_2^2),\\
&\partial_{s}\frac{\mu_1B(s\star u,p)+\mu_2B(s\star v,p)}{2p}=\delta_{p}e^{2p\delta_{p}s}(\mu_1B(u,p)+\mu_2B(v,p)),
\end{aligned}
\right.
\end{align*}
by (3), it follows that
\begin{align*}
&(\|\nabla \bar{u}_n\|_2^2+\nabla \bar{v}_n\|_2^2)
-\delta_{p}(\mu_1B(\bar{u}_n,p)+\mu_2B(\bar{v}_n,p))\\
+&\int_{\mathbb{R}^N}(x\cdot\nabla \beta(x))\bar{u}_n(x)\bar{v}_n(x)dx\rightarrow 0~~\text{as~}n\rightarrow\infty.
\end{align*}
Thus we finish the proof.
\end{proof}

Next lemma gives the boundedness of Palais-Smale sequence $\{(\bar{u}_n,\bar{v}_n)\}$.

\begin{lemma}\label{lbdd}
Suppose that $\beta(x)$ and $x\cdot\nabla \beta(x)$ are bounded. Then $\{(\bar{u}_n,\bar{v}_n)\}$ is bounded in $H^1_{r}(\mathbb{R}^N)\times H^1_{r}(\mathbb{R}^N)$. Furthermore, there exists a constant $\bar{C}>0$ such that
\begin{align*}
\|\nabla \bar{u}_n\|_2^2+\|\nabla \bar{v}_n\|_2^2\ge \bar{C}.
\end{align*}
\end{lemma}
\begin{proof}
By \eqref{add condition}, it yields that
\begin{align*}
2p\delta_{p}c\leftarrow&2p\delta_{p}E(\bar{u}_n,\bar{v}_n)-D\bar{E}(s_n,u_n,v_n)[(1,0,0)]\\
=&(p\delta_{p}-1)(\|\nabla \bar{u}_n\|_2^2+\|\nabla \bar{v}_n\|_2^2)
-\int_{\mathbb{R}^N}(2p\delta_{p}\beta(x)+x\cdot\nabla\beta(x))\bar{u}_n(x)\bar{v}_n(x)dx
\end{align*}
as $n\rightarrow\infty$.
Since $1+\frac{\alpha+2}{N}<p<\frac{\alpha+N}{N-2}$, we have  $p\delta_{p}-1>0$. By the boundedness of $\beta(x)$ and $x\cdot\nabla\beta(x)$, there exist $C_1>0$ and $C_2>0$ such that
\begin{align*}
C_2\le\|\nabla \bar{u}_n\|_2^2+\|\nabla \bar{v}_n\|_2^2\le C_1.
\end{align*}
Hence the desired results is obtained.
\end{proof}

Since $E'|_{S_{\xi}\times S_{\eta}}(\bar{u}_n,\bar{v}_n)\rightarrow 0$, there exist two sequence $\{\lambda_{1n}\}\subset\mathbb{R}$ and $\{\lambda_{2n}\}\subset\mathbb{R}$ such that
\begin{align}\label{dvt e}
\begin{split}
&\int_{\mathbb{R}^N}[\nabla \bar{u}_n\nabla \varphi+\nabla \bar{v}_n\nabla \psi]dx
-\mu_1\int_{\mathbb{R}^N}(I_{\alpha}\star(|\bar{u}_n|^p))|\bar{u}_n|^{p-2}\bar{u}_n\varphi dx\\
-&\mu_2\int_{\mathbb{R}^N}(I_{\alpha}\star(|\bar{v}_n|^p))|\bar{v}_n|^{p-2}\bar{v}_n\psi dx
-\int_{\mathbb{R}^N}\beta(x)(\bar{u}_n\psi+\bar{v}_n\varphi)dx\\
+&\lambda_{1n}\int_{\mathbb{R}^N}\bar{u}_n\varphi dx+\lambda_{2n}\int_{\mathbb{R}^N}\bar{v}_n\psi dx
=o_{n}(1)\|(\varphi,\psi)\|_{H^1(\mathbb{R}^N)\times H^1(\mathbb{R}^N)}
\end{split}
\end{align}
for any $(\varphi,\psi)\in H^1(\mathbb{R}^N)\times H^1(\mathbb{R}^N)$ with $o_{n}(1)\rightarrow 0$ as $n\rightarrow\infty$.

\begin{lemma}\label{llambda}
Both $\{\lambda_{1n}\}$ and $\{\lambda_{2n}\}$ are bounded sequences. Moreover,  if $2\beta(x)+\frac{x\cdot\nabla \beta(x)}{\delta_{p}}\ge0$, then at least one of the sequences converges to a strict positive value.
\end{lemma}
\begin{proof}
Taking $(\varphi,\psi)=(\bar{u}_n,0)$ ($(\varphi,\psi)=(0,\bar{v}_n)$, respectively) in \eqref{dvt e}, we can obtain that
\begin{align}\label{l1}
-\lambda_{1n}\xi^2+o_{n}(1)=\|\nabla \bar{u}_n\|_2^2-\mu_1B(\bar{u}_n,p)-\int_{\mathbb{R}^N}\beta(x)\bar{u}_n\bar{v}_n dx,
\end{align}
and
\begin{align}\label{l2}
-\lambda_{2n}\eta^2+o_{n}(1)=\|\nabla \bar{v}_n\|_2^2-\mu_2B(\bar{v}_n,p)-\int_{\mathbb{R}^N}\beta(x)\bar{u}_n\bar{v}_n dx.
\end{align}
From the boundedness of $\beta(x)$ and $\{(\bar{u}_n,\bar{v}_n)\}$ in $H^1_{r}(\mathbb{R}^N)\times H^1_{r}(\mathbb{R}^N)$, it is easy to deduce that $\{\lambda_{1n}\}$ and $\{\lambda_{2n}\}$ are bounded.

By \eqref{add condition}, \eqref{l1} and \eqref{l2}, it yields that
\begin{align*}
&\lambda_{1n}\xi^2+\lambda_{2n}\eta^2+o_{n}(1)\\
=&-(\|\nabla \bar{u}_n\|_2^2+\|\nabla \bar{v}_n\|_2^2)
+(\mu_1B(\bar{u}_n,p)+\mu_2B(\bar{v}_n,p))+\int_{\mathbb{R}^N}2\beta(x)\bar{u}_n\bar{v}_n dx\\
=&\bigg(\frac{1}{\delta_{p}}-1\bigg)(\|\nabla \bar{u}_n\|_2^2+\|\nabla \bar{v}_n\|_2^2)+\int_{\mathbb{R}^N}\bigg(2\beta(x)+\frac{x\cdot\nabla \beta(x)}{\delta_{p}}\bigg)\bar{u}_n\bar{v}_n dx.
\end{align*}
Since $\delta_{p}<1$ and $2\beta(x)+\frac{x\cdot\nabla \beta(x)}{\delta_{p}}\ge0$, by Lemma \ref{lbdd}, it can be inferred that
\begin{align*}
\lambda_{1n}\xi^2+\lambda_{2n}\eta^2+o_{n}(1)\ge\bigg(\frac{1}{\delta_{p}}-1\bigg)C_2>0
\end{align*}
for $n$ large. Hence the lemma holds.
\end{proof}
From the above lemma, there exists a subsequence of $\{\lambda_{1n}\}$ (resp. $\{\lambda_{2n}\}$) and $\lambda_1\in\mathbb{R}$ (resp. $\lambda_2\in\mathbb{R}$)  such that $\lambda_{1n}\rightarrow\lambda_1$ (resp. $\lambda_{2n}\rightarrow\lambda_2$) as $n\rightarrow\infty$. The signs of $\lambda_1$ and $\lambda_2$  play an important role in the strong convergence of $\bar{u}_n,~\bar{v}_n$ in $H^1(\mathbb{R}^N)$.
\begin{lemma}\label{llp}
If $\lambda_1>0~(\text{resp.~} \lambda_2>0)$, then $\bar{u}_n\rightarrow\bar{u}~(\text{resp.~}\bar{v}_n\rightarrow\bar{v}) $ strongly in $H^1(\mathbb{R}^N)$.
\end{lemma}
\begin{proof}
By Lemma \ref{lbdd},  $\{(\bar{u}_n,\bar{v}_n)\}$ is bounded in $H^1_{r}(\mathbb{R}^N)\times H^1_{r}(\mathbb{R}^N)$. By \eqref{bu}, there exists a $(\bar{u},\bar{v})\in H^1_{r}(\mathbb{R}^N)\times H^1_{r}(\mathbb{R}^N)$ such that
\begin{align}\label{chlp}
\left\{\begin{aligned}
&(\bar{u}_n,\bar{v}_n)\rightharpoonup(\bar{u},\bar{v})\text{~weakly in }H^1_{r}(\mathbb{R}^N)\times H^1_{r}(\mathbb{R}^N),\\
&B(\bar{u}_n,p)\rightarrow B(\bar{u},p),~B(\bar{v}_n,p)\rightarrow B(\bar{v},p),\\
&(\bar{u}_n,\bar{v}_n)\rightarrow(\bar{u},\bar{v})\text{~a.e. in }\mathbb{R}^N.
\end{aligned}
\right.
\end{align}
From \eqref{dvt e}, we have
\begin{align*}
\langle E'(\bar{u}_n,\bar{v}_n)-\lambda_{1n}(\bar{u}_n,0),(\bar{u}_n,0)\rangle\rightarrow\langle E'(\bar{u},\bar{v})-\lambda_{1}(\bar{u},0),(\bar{u},0)\rangle~~~\text{as }n\rightarrow\infty.
\end{align*}
Combining with \eqref{chlp}, it holds that
\begin{align*}
&\|\nabla \bar{u}_n\|_2^2+\lambda_{1n}\|\bar{u}_n\|_2^2-\int_{\mathbb{R}^N}\beta(x)\bar{u}_n\bar{v}_n
\rightarrow\|\nabla \bar{u}\|_2^2+\lambda_{1}\|\bar{u}\|_2^2-\int_{\mathbb{R}^N}\beta(x)\bar{u}\bar{v}.
\end{align*}
From $\beta(x)$ is sufficiently  small and $\lambda_1>0$, we can obtain that
\begin{align}\label{abc}
\|\nabla \bar{u}_n\|_2^2+\lambda_{1n}\|\bar{u}_n\|_2^2
\rightarrow\|\nabla \bar{u}\|_2^2+\lambda_{1}\|\bar{u}\|_2^2.
\end{align}
From the low semi-continuous of $L^2$-norm and Fatou's Lemma, it follows that
\begin{align*}
\lim\inf_{n\rightarrow\infty}\|\nabla \bar{u}_n\|_2^2\ge\|\nabla \bar{u}\|_2^2,~
\lim\inf_{n\rightarrow\infty}\|\bar{u}_n\|_2^2\ge\|\bar{u}\|_2^2,
\end{align*}
combining with \eqref{abc}, and due to $\lambda_1>0$, hence
\begin{align*}
\lim_{n\rightarrow\infty}\|\nabla \bar{u}_n\|_2^2=\|\nabla \bar{u}\|_2^2~\text{and}~\lim_{n\rightarrow\infty}\|\bar{u}_n\|_2^2=\|\bar{u}\|_2^2,
\end{align*}
For $\lambda_2>0$, the proof is similar. Thus we complete the proof.
\end{proof}

\begin{proof}[\textbf{Proof of Theorem \ref{t2}}]
From Lemma \ref{llambda}, we may assume that $\lambda_1>0$. By Lemma \ref{llp}, $\bar{u}_n\rightarrow\bar{u}$ strongly in $H^r_1(\mathbb{R}^N)$. If $\lambda_2>0$, then  $\bar{v}_n\rightarrow\bar{v}$ strongly in $H^r_1(\mathbb{R}^N)$. Hence we remain to show $\lambda_2>0$.  We argue it by contradiction and assume that $\lambda_2\le0$. By the weak limit and \eqref{dvt e},  $\bar{v}$ satisfies
\begin{align*}
-\Delta\bar{v}=-\lambda_2 \bar{v}+\mu_2(I_{\alpha}\star v)|\bar{v}|^{q-2}\bar{v}+\beta(x)\bar{u}\bar{v}\ge0,
\end{align*}
due to $\bar{u},~\bar{v}\ge0$ and $\beta(x)\ge0$.
By Liouville-type Lemma in \cite{I14}, we get that $\bar{v}\equiv0$. However, $(\bar{u},0)$ can not be the solution of \eqref{system}, hence there exists a contradiction. Therefore, $\lambda_2>0$. By Lemma \ref{llp}, $\bar{v}_n\rightarrow\bar{v}$ strongly in $H^1_{r}(\mathbb{R}^N)$. Moreover, by maximum principle, $u,~v>0$. Thus the theorem is established.
\end{proof}


\section{$V(x)$ for $\max\{1+\frac{\alpha}{N},2\}<p,q<1+\frac{\alpha+2}{N}$}
In this section, we consider the $L^2$-subcritical case  for \eqref{system} with different potentials.
\subsection{Proof of Theorem \ref{tv2}}
In this subsection, we consider the case $V_1(x)$ and $V_2(x)$  satisfy \textbf{(V2)}. 
\begin{lemma}\cite{R93}\label{lcompact}
Let $V_1(x)$ and $V_2(x)$  satisfy \textbf{(V2)}. Then $\tilde{H}_{1}\times\tilde{H}_{2}\hookrightarrow L^{p}\times L^{q}$ for $p,q\in[2,2^{*})$ is compact.
\end{lemma}
Similar to Lemma \ref{lbfb}, we obtain that
\begin{align*}
E_{V}(u,v)\ge& \frac12(\|\nabla u\|_2^2+\|\nabla v\|_2^2)-\frac{\mu_1C(N,\alpha,p)\xi^{2p(1-\delta_{p})}}{2p}\|\nabla u\|_2^{2p\delta_p}\\
-&\frac{\mu_2C(N,\alpha,q)\eta^{2q(1-\delta_{q})}}{2q}\|\nabla v\|_2^{2q\delta_q}
-\beta \xi\eta+\frac12\int_{\mathbb{R}^N}(V_1(x)|u|^2+V_2(x)|v|^2)dx.
\end{align*}
Since $1+\frac{\alpha}{N}<p,q<1+\frac{\alpha+2}{N}$ and $(u,v)\in\tilde{H_{1}}\times\tilde{H_{2}}$, $E_{V}(u,v)$ is also bounded from below. Furthermore, there exists a bounded minimizing sequence $\{(u_n,v_n)\}\in\tilde{H_{1}}\times\tilde{H_{2}}$.
\begin{proof}[\textbf{Proof of theorem \ref{tv2}}]
Notice that $\{(u_n,v_n)\}\in\tilde{H_{1}}\times\tilde{H_{2}}$ is bounded. Hence there exists a $(u_0,v_0)\in\tilde{H_{1}}\times\tilde{H_{2}}$ such that $(u_n,v_n)\rightharpoonup(u_0,v_0)$ weakly in $\tilde{H_{1}}\times\tilde{H_{2}}$. By Lemma \ref{lcompact} and \eqref{bu}, we can obtain that $B(u_n,p)\to B(u_0,p)$ and $B(v_n,q)\to B(v_0,q)$. Moreover, $(u_n,v_n)\to(u_0,v_0)$ strongly in $\tilde{H_{1}}\times\tilde{H_{2}}$ due to
\begin{align*}
e_{V}(\xi,\eta)\le E_{V}(u_0,v_0)\le\liminf_{n\rightarrow\infty}E_{V}(u_n,v_n)=e_{V}(\xi,\eta).
\end{align*}
Thus we establish the theorem.

\end{proof}

\subsection{Proof of Theorem \ref{tv1}}
In this subsection, we focus on the case that $V_1(x)$ and $V_2(x)$  satisfy \textbf{(V1)}. 
\begin{lemma}\label{evsubadd}
Assume that $V_1(x)$ and $V_2(x)$  satisfy \textbf{(V1)}. Then there holds
\begin{itemize}
  \item[(i)] $e_{V}(\xi,\eta)\le e(\xi,\eta)\le0$; if $\xi+\eta>0$, then  $e_{V}(\xi,\eta)< e(\xi,\eta)<0$.
  \item[(ii)] every minimizing sequence for $e_{V}(\xi,\eta)$ is bounded in $H^1(\mathbb{R}^N)\times H^1(\mathbb{R}^N)$,
  \item[(iii)] $e_{V}(\xi,\eta)$ is continuous about $(\xi,\eta)$. Furthermore, $e_{V}(\xi,\eta)\le e_{V}(\xi-\xi_1,\eta-\eta_1)+e(\xi_1,\eta_1)$ for $\xi_1>0$ and $\eta_1>0$,
  \item[(iv)]$e_{V}(\xi,\eta)$ is nondecreasing with respect to $(\xi,\eta)$.
\end{itemize}
\end{lemma}
\begin{proof}
(i) holds immediately from $V_1(x),~V_2(x)<0$. And the proof of  (ii) is similar to the proof of Lemma \ref{lbfb}. If (iii) holds, combining with (i), then (iv) holds.  Hence we only need to show that (iii) holds.

Assume that $|\xi_n-\xi|+|\eta_n-\eta|=o_n(1)$,  for any $\varepsilon>0$, there exists a sequence $\{(u_n,v_n)\}\in S_{\xi_n}\times S_{\eta_n}$  such that
\begin{align}\label{Eon}
 E_{V}(u_n,v_n)\le e_{V}(\xi_n,\eta_n)+o_n(1).
\end{align}
Let $\hat{u}=\frac{\xi}{\xi_n}u_n$ and $\hat{v}=\frac{\eta}{\eta_n}v_n$. Then $(\hat{u},\hat{v})\in S_{\xi}\times S_{\eta}$ and
\begin{align*}
E_{V}(\hat{u},\hat{v})
=&\frac{1}{2}\bigg(\frac{\xi^2}{\xi_n^2}\|\nabla u_n\|^2_2+\frac{\eta^2}{\eta_n^2}\|\nabla v_n\|^2_2\bigg)-\frac{\xi}{\xi_n}\frac{\eta}{\eta_n}\int_{\mathbb{R}^N}\beta u_nv_ndx\\
&+\frac12\int_{\mathbb{R}^N}\bigg(\frac{\xi^2}{\xi_n^2}V_1(x) |u_n|^2+\frac{\eta^2}{\eta_n^2}V_2(x) |v_n|^2\bigg)dx\\
&-\frac{\xi^{2p}}{\xi_n^{2p}}\frac{\mu_1}{2p}B(u_n,p)-\frac{\eta^{2q}}{\eta_n^{2q}}\frac{\mu_2}{2q}B(v_n,q),
\end{align*}
by \eqref{Eon} and $|\xi_n-\xi|+|\eta_n-\eta|=o_n(1)$, it follows that
\begin{align*}
e_{V}(\xi,\eta)\le E_{V}(\hat{u},\hat{v})=E_{V}(u_n,v_n)+o_n(1)\le e_{V}(\xi_n,\eta_n)+o_n(1),
\end{align*}
consequently, $e_{V}(\xi,\eta)\le e_{V}(\xi_n,\eta_n)+o_n(1)$. Vice versa. Hence $e_{V}(\xi,\eta)$ is continuous about $(\xi,\eta)$.

Next we show the subadditivity of $e_{V}(\xi,\eta)$, which is similar to Lemma \ref{lsubadd}.

For any $\varepsilon>0$, we can seek $(\varphi_{1\varepsilon},\psi_{1\varepsilon}),(\varphi_{1\varepsilon},\psi_{1\varepsilon})\in C_{c}(\mathbb{R}^N)\times C_{c}(\mathbb{R}^N)$ such that
\begin{align}\label{vpn}
\begin{split}
&(\varphi_{1\varepsilon},\psi_{1\varepsilon})\in S_{\xi-\xi_1}\times S_{\eta-\eta_1}, ~E_{V}(\varphi_{1\varepsilon},\psi_{1\varepsilon})\le e_{V}(\xi-\xi_1,\eta-\eta_1)+\varepsilon,\\
&(\varphi_{2\varepsilon},\psi_{2\varepsilon})\in S_{\xi_1}\times S_{\eta_1}, ~E(\varphi_{2\varepsilon},\psi_{2\varepsilon})\le e(\xi_1,\eta_1)+\varepsilon,
\end{split}
\end{align}
Set $(u_{\varepsilon,n}(x),v_{\varepsilon,n}(x)):=(\varphi_{1\varepsilon}(x)+\varphi_{2\varepsilon}(x-n\textbf{e}_1),\psi_{1\varepsilon}(x)+\psi_{2\varepsilon}(x-n\textbf{e}_1))$,
where $\textbf{e}_1=(1,0,...,0)$. Since $(\varphi_{1\varepsilon},\psi_{1\varepsilon})$ and $(\varphi_{2\varepsilon},\psi_{2\varepsilon})$ have compact support, there holds that
 $$(u_{\varepsilon,n}(x),v_{\varepsilon,n}(x))\in S_{\xi}\times S_{\eta}$$
for large $n$.  Therefore,
\begin{align*}
e_{V}(\xi,\eta)\le E_{V}(u_{\varepsilon,n}(x),v_{\varepsilon,n}(x))=E_{V}(\varphi_{1\varepsilon}(x),\psi_{1\varepsilon}(x))+E_{V}(\varphi_{2\varepsilon}(x-n\textbf{e}_1),\psi_{2\varepsilon}(x-n\textbf{e}_1)).
\end{align*}
Due to $\lim\limits_{|x|\rightarrow\infty}V_1(x)=\lim\limits_{|x|\rightarrow\infty}V_2(x)=0$, it means that $E_{V}(\varphi_{2\varepsilon}(x-n\textbf{e}_1),\psi_{2\varepsilon}(x-n\textbf{e}_1))\rightarrow E(\varphi_{2\varepsilon},\psi_{2\varepsilon})$ as $n\rightarrow\infty$. Together with \eqref{vpn}, there holds that
\begin{align*}
e_{V}(\xi,\eta)
\le & E_{V}(\varphi_{1\varepsilon}(x),\psi_{1\varepsilon}(x))+E(\varphi_{2\varepsilon},\psi_{2\varepsilon})\le e_{V}(\xi-\xi_1,\eta-\eta_1)+e(\xi_1,\eta_1)+2\varepsilon
\end{align*}
for any $\varepsilon>0$.  Due to the arbitrary of $\varepsilon>0$, hence (iii) holds.
\end{proof}

\begin{lemma}\label{sub-lsplit}
Let $\{(u_n,v_n)\}\subset S_{\xi}\times S_{\eta}$ be a minimizing sequence for $e_{V}(\xi,\eta)$ and there exists $(u_0,v_0)\in H^1(\mathbb{R}^N)\times H^1(\mathbb{R}^N)$  such that $(u_n,v_n)\rightharpoonup (u_0,v_0)$ in $H^1(\mathbb{R}^N)\times H^1(\mathbb{R}^N) $.  Set $\xi_1=\|u_0\|_2,~\eta_1=\|v_0\|_2$. If $\xi_1^2+\eta_1^2<\xi^2+\eta^2$, then there exist $\{y_n\}\subset\mathbb{R}^N$ and $(\varphi_0,\psi_0)\in H^1(\mathbb{R}^N)\times H^1(\mathbb{R}^N)\setminus{(0,0)}$ such that
\begin{align}\label{yninfty}
\left\{\begin{aligned}
&|y_n|\rightarrow\infty,\\
&(u_n(\cdot+y_n),v_n(\cdot+y_n))\rightharpoonup(\varphi_0,\psi_0)\text{~weakly in~}H^1(\mathbb{R}^N)\times H^1(\mathbb{R}^N),
\end{aligned}
\right.\text{~~as~}n\rightarrow\infty,
\end{align}
and
\begin{align}\label{unu0vnv0}
\lim_{n\rightarrow\infty}\|u_n-u_0-\varphi_0(\cdot-y_n)\|_2^2=0,~\lim_{n\rightarrow\infty}\|v_n-v_0-\psi_0(\cdot-y_n)\|_2^2=0,
\end{align}
where $\|\varphi_0\|_2^2=\xi_2^2=\xi^2-\xi_1^2,~\|\psi_0\|_2^2=\eta_2^2=\eta^2-\eta_1^2$. Furthermore,
\begin{align}\label{eveta}
e_{V}(\xi,\eta)=e_{V}(\xi_1,\eta_1)+e(\xi_2,\eta_2),
\end{align}
and
\begin{align}\label{ju0}
E_{V}(u_0,v_0)=e_{V}(\xi_1,\eta_1),~E(\varphi_0,\psi_0)=e(\xi_2,\eta_2).
\end{align}
\end{lemma}

\begin{proof} We divide the proof into three steps.

 \textbf{Step 1: We find $\{y_{n}\}\subset\mathbb{R}^{N}$ and $(\varphi_0,\psi_{0})$ such that \eqref{yninfty} holds}.

We claim that
\begin{align*}
\delta_{0}:=\liminf _{n \rightarrow \infty} \sup _{y \in \mathbb{R}^{N}} \int_{y+B(0,1)}[|u_{n}-u_{0}|^2+|v_{n}-v_{0}|^2] dx>0,
\end{align*}
where $B(0,1):=\{x \in \mathbb{R}^{N}:|x|\le1\}$. The claim can be shown by contradiction. Suppose that  $\delta_{0}=0$. Then $(u_n,v_n)\to (u_0,v_0)$ strongly in $L^{s}(\mathbb{R}^N)\times L^{t}(\mathbb{R}^N)$ for $s,t\in(2,2^{*})$ (see \cite{W96}). By \eqref{bu}, we have $B(u_n,p)\to B(u_0,p)$ and $B(v_n,q)\to B(v_0,q)$. Since $(u_n,v_n)\rightharpoonup (u_0,v_0)$ weakly in $H^1(\mathbb{R}^N)\times H^1(\mathbb{R}^N)$, we can verify that $$\int_{\mathbb{R}^{N}}(V_1(x)(u_{n}-u_{0})^2+V_2(x)(v_{n}-v_{0})^2)dx\rightarrow 0.$$
Thus,
\begin{align*}
e_{V}(\xi,\eta)&=E_{V}(u_n,v_n)+o(1) \\
&\ge
E_{V}(u_0,v_0)+E_{V}(u_n-u_0,v_n-v_0)+o(1) \\
&=E_{V}(u_0,v_0)+\frac{1}{2}(\|\nabla(u_n-u_0)\|_2^2+\|\nabla(v_n-v_0)\|_2^2)\\
&~~-\int_{\mathbb{R}^{N}}\beta(u_n-u_0)(v_n-v_0)dx+o(1)\\
&\ge E_{V}(u_0,v_0)-\beta\|(u_n-u_0)\|_2\|(v_n-v_0)\|_2\\
&=E_{V}(u_0,v_0)-\beta\xi_2\eta_2,
\end{align*}
together with \eqref{e2}, we can obtain that
\begin{align*}
e_{V}(\xi,\eta)> E_{V}(u_0,v_0)+e(\xi_2,\eta_2)+o(1)\ge e_{V}(\xi_1,\eta_1)+e(\xi_2,\eta_2) ,
\end{align*}
which is a contradiction with Lemma \ref{evsubadd} (iii). Hence $\delta_0>0$.

From $\delta_0>0$ and $(u_n,v_n)\rightarrow (u_0,v_0)$ in $L_{loc}^{2}(\mathbb{R}^{N})\times L_{loc}^{2}(\mathbb{R}^{N})$, we can find a sequence $\{y_n\} \subset \mathbb{R}^{N}$ such that
\begin{align*}
\int_{y_{n}+B(0,1)}(|u_n-u_0|^{2}+|v_n-v_0|^{2})dx\rightarrow c_0>0~\text{ and}~~ |y_{n}| \rightarrow \infty,\text{~as~}n\to\infty.
\end{align*}
Denote that $(u_n(\cdot+y_n),v_n(\cdot+y_n))\rightharpoonup(\varphi_0,\psi_0)$ weakly in $H^1(\mathbb{R}^N)\times H^1(\mathbb{R}^N)$.  Due to $c_{0}>0$, $(\varphi_0,\psi_0)\neq(0,0)$. Therefore, $\{y_{n}\}$ and $(\varphi_0,\psi_0)$ satisfy \eqref{yninfty}.

Since $y_{n} \rightarrow \infty$, we have
\begin{align*}
&\|u_n-u_0-\varphi_0(\cdot-y_n)\|_2^2\\
 =&\|u_n\|_2^2+\|u_0\|_2^2+\|\varphi_0\|_2^2-2\langle u_n, u_0\rangle_{L^2}-2\langle u_n(\cdot+y_n), \varphi_0\rangle_{L^2}+o(1) \\
=&\|u_n\|_2^2-\|u_0\|_2^2-\|\varphi_0\|_2^2+o(1),
\end{align*}
and
\begin{align*}
&\|v_n-v_0-\psi_0(\cdot-y_n)\|_2^2 \\
=&\|v_n\|_2^2+\|v_0\|_2^2+\|\psi_0\|_2^2-2\langle v_n, v_0\rangle_{L^2}-2\langle v_n(\cdot+y_n), \psi_0\rangle_{L^2}+o(1) \\
=&\|v_n\|_2^2-\|v_0\|_2^2-\|\psi_0\|_2^2+o(1),
\end{align*}
hence it yields that
\begin{align*}
&\xi_2^2:=\|\varphi_0\|_2^2\le\liminf_{n\rightarrow\infty}(\|u_n\|^2_2-\|u_0\|^2_2)=\xi^2-\xi_1^2,\\
&\eta_2^2:=\|\psi_0\|_2^2\le\liminf_{n\rightarrow\infty}(\|u_n\|^2_2-\|u_0\|^2_2)=\eta^2-\eta_1^2.
\end{align*}
From $c_0>0$, it follows that $\xi_2^2+\eta_2^2>0$.

\textbf{Step 2: We show $\{y_{n}\} \subset \mathbb{R}^{N}$ and $(\varphi_0,\psi_0)$ such that \eqref{unu0vnv0} holds}.

Let $\delta_1^2=\lim\limits_{n\rightarrow\infty}\|u_n-u_0-\varphi_0(\cdot-y_n)\|_2^2$ and $\delta_2^2=\lim\limits_{n\rightarrow\infty}\|v_n-v_0-\psi_0(\cdot-y_n)\|_2^2$. Then $\delta_1^2=\xi^2-\xi_1^2-\xi_2^2$ and $\delta_2^2=\eta^2-\eta_1^2-\eta_2^2$. We next show that $\delta_1=0$ and $\delta_2=0$ by contradiction. Suppose that $\delta_1^2+\delta_2^2>0$. Then
\begin{align*}
&\quad\|\nabla u_n\|_2^2-\|\nabla u_0\|_2^2-\|\nabla\varphi_0\|_2^2-\|\nabla (u_n-u_0-\varphi_0(\cdot-y_n))\|_2^2\\
&=-2\|\nabla u_0\|_2^2-2\|\nabla\varphi_0\|_2^2+2\langle \nabla u_n, \nabla u_0\rangle+2\langle \nabla u_n(\cdot+y_n), \nabla \varphi_0\rangle\\
&=o(1).
\end{align*}
Similarly,
\begin{align*}
&\quad\|\nabla v_n\|_2^2-\|\nabla v_0\|_2^2-\|\nabla\psi_0\|_2^2-\|\nabla (v_n-v_0-\psi_0(\cdot-y_n))\|_2^2\\
&=-2\|\nabla v_0\|_2^2-2\|\nabla\psi_0\|_2^2+2\langle \nabla v_n, \nabla v_0\rangle+2\langle \nabla v_n(\cdot+y_n), \nabla \psi_0\rangle\\
&=o(1).
\end{align*}
Furthermore,
\begin{align*}
&\frac12\int_{\mathbb{R}^N}V_1(x)(|u_n|^2-| u_0|_2^2-|\varphi_0|^2-|(u_n-u_0-\varphi_0(\cdot-y_n))|^2)dx=o(1),\\
&\frac12\int_{\mathbb{R}^N}V_2(x)(|v_n|^2-| v_0|_2^2-|\psi_0|^2-|(v_n-v_0-\psi_0(\cdot-y_n))|^2)dx=o(1).
\end{align*}
In addition, by Br\'ezis-Lieb Lemma, there holds that
\begin{align*}
B(u_n,p)=&B(u_0,p)+B(u_n-u_0,p)+o(1)\\=&B(u_0,p)+B(\varphi_0,p)+B(u_n-u_0-\varphi_0(\cdot-y_n),p)+o(1),
\end{align*}
and
\begin{align*}
B(v_n,q)&=B(v_0,q)+B(v_n-v_0,q)+o(1)\\
=&B(v_0,q)+B(\psi_0,q)+B(v_n-v_0-\psi_0(\cdot-y_n),q)+o(1).
\end{align*}
As a consequence,
\begin{align*}
E_{V}(u_n,v_n)=&E_{V}(u_0,v_0)+E_{V}(\varphi_0,\psi_0)\\
&-E_{V}(u_n-u_0-\varphi_0(\cdot-y_n),v_n-v_0-\psi_0(\cdot-y_n))+o(1).
\end{align*}
Since $(u_n,v_n)\rightharpoonup (u_0,v_0)$ weakly in $H^1(\mathbb{R}^N)\times H^1(\mathbb{R}^N)$, $|y_n|\rightarrow\infty$ and $$\lim_{|x|\rightarrow\infty}V_1(x)=\lim_{|x|\rightarrow\infty}V_2(x)=0,$$ we have
\begin{align}\label{v1v20}
\int_{\mathbb{R}^N}V_1(x)|u_n-u_0-\varphi_0(x-y_n)|^2+V_2(x)|v_n-v_0-\psi_0(x-y_n)|^2dx\rightarrow 0.
\end{align}
Notice that
\begin{align*}
&E_{V}(u_n-u_0-\varphi_0(\cdot-y_n),v_n-v_0-\psi_0(\cdot-y_n))\\
=&E(u_n-u_0-\varphi_0(\cdot-y_n),v_n-v_0-\psi_0(\cdot-y_n))\\
&+\frac12\int_{\mathbb{R}^N}V_1(x)|u_n-u_0-\varphi_0(\cdot-y_n)|^2+V_2(x)|v_n-v_0-\psi_0(\cdot-y_n)|^2dx,
\end{align*}
together with \eqref{v1v20}, it yields that
\begin{align*}
&\liminf_{n\rightarrow\infty}E_{V}(u_n-u_0-\varphi_0(\cdot-y_n),v_n-v_0-\psi_0(\cdot-y_n))\\
=&\liminf_{n\rightarrow\infty}E(u_n-u_0-\varphi_0(\cdot-y_n),v_n-v_0-\psi_0(\cdot-y_n))\\
\ge &e(\delta_1,\delta_2),
\end{align*}
Similarly, we have
\begin{align*}
\liminf_{n\rightarrow\infty}E_{V}(\varphi_0(\cdot-y_n),\psi_0(\cdot-y_n))\ge e(\xi_2,\eta_2).
\end{align*}
Consequently,
\begin{align*}
e_{V}(\xi,\eta)\ge e_{V}(\xi_1,\eta_1)+e(\xi_2,\eta_2)+e(\delta_1,\delta_2).
\end{align*}
By Remark \ref{r31}, we have $e(\xi_2+\delta_1,\eta_2+\delta_2)<e(\xi_2,\eta_2)+e(\delta_1,\delta_2)$ for $\delta_1^2+\delta_2^2>0$. By Lemma \ref{evsubadd} (iii), it yields that
\begin{align*}
e_{V}(\xi,\eta)&\ge e_{V}(\xi_1,\eta_1)+e(\xi_2,\eta_2)+e(\delta_1,\delta_2)\\
&> e_{V}(\xi_1,\eta_1)+e(\xi_2+\delta_1,\eta_2+\delta_2)\\
&\ge e_{V}(\xi_1+\xi_2+\delta_1,\eta_1+\eta_2+\delta_2)\\
&=e_{V}(\xi,\eta),
\end{align*}
which is a contradiction. Hence $\delta_1^2+\delta_2^2=0$.

\textbf{Step 3: We show $\{y_{n}\} \subset \mathbb{R}^{N}$ and $(\varphi_0,\psi_0)$ satisfy \eqref{eveta} and \eqref{ju0}}.

Since $\delta_1=\delta_2=0$, we have
\begin{align*}
e_{V}(\xi,\eta)=\lim_{n\rightarrow\infty}E_{V}(u_n,v_n)
&=\lim_{n\rightarrow\infty}\big(E_{V}(u_0,v_0)+E_{V}(\varphi_0(\cdot+y_n),\psi_0(\cdot+y_n))\big)\\
&\ge E_{V}(u_0,v_0)+E(\varphi_0,\psi_0)\\
&\ge e_{V}(\xi_1,\eta_1)+e(\xi_2,\eta_2),
\end{align*}
together with  Lemma \ref{evsubadd} (iii), there holds that
$$
e_{V}(\xi,\eta)=e_{V}(\xi_1,\eta_1)+e(\xi_2,\eta_2),
$$
Furthermore, $E_{V}(u_0,v_0)=e_{V}(\xi_1,\eta_1)$ and $E(\varphi_0,\psi_0)=e(\xi_2,\eta_2)$. Thus we finish the proof.
\end{proof}

\begin{proof}[\textbf{Proof of Theorem \ref{tv1}}]
 Let $\{(u_n,v_n)\}$ be the  minimizing sequence for $e_{V}(\xi,\eta)$. We divide the proof into two steps.

\textbf{Step 1: $\{(u_n,v_n)\}$ is strongly convergent in $L^2(\mathbb{R}^N)\times L^2(\mathbb{R}^N)$}.

By  Lemma \ref{evsubadd} (ii), we can find that $\{(u_n,v_n)\}$ is bounded in $H^1(\mathbb{R}^N)\times H^1(\mathbb{R}^N)$. Hence there exists a $(u_0,v_0)\in H^1(\mathbb{R}^N)\times H^1(\mathbb{R}^N)$ such that
\begin{align*}
(u_n,v_n)\rightharpoonup (u_0,v_0) \text{ weakly in }H^1(\mathbb{R}^N)\times H^1(\mathbb{R}^N).
\end{align*}
Since $\{(u_n,v_n)\}$ is the minimizing sequence for $e_{V}(\xi,\eta)$, we have $dE_{V}|_{S_{\xi}\times S_{\eta}}(u_n,v_n)\rightarrow 0$. Moreover, there exist two sequences  $\lambda_{1n}\subset\mathbb{R}$ and $\lambda_{2n}\subset\mathbb{R}$ such that
\begin{align}\label{evd}
\begin{split}
&\int_{\mathbb{R}^N}(\nabla u_n\nabla \varphi+\nabla v_n\nabla \psi)dx+\int_{\mathbb{R}^N}(V_1(x)u_n\varphi+V_2(x)v_n\psi) dx\\
-&\int_{\mathbb{R}^N}(\mu_1(I_{\alpha}\star|u_n|^{p})|u_n|^{p-2}u_n\varphi+\mu_2(I_{\alpha}\star|v_n|^{q})|v_n|^{q-2}v_n\psi)dx\\
-&\int_{\mathbb{R}^N}\beta(x)(u_n\psi+v_n\varphi)dx+\lambda_{1n}\int_{\mathbb{R}^N}u_n\varphi dx+\lambda_{2n}\int_{\mathbb{R}^N}v_n\psi dx\\
=&o_{n}(1)\|(\varphi,\psi)\|_{H^1(\mathbb{R}^N)\times H^1(\mathbb{R}^N)}
\end{split}
\end{align}
for any $(\varphi,\psi)\in H^1(\mathbb{R}^N)\times H^1(\mathbb{R}^N)$ with $o_{n}(1)\rightarrow 0$ as $n\rightarrow\infty$.
Taking $(\varphi,\psi)=(u_n,0)$ and $(\varphi,\psi)=(0,v_n)$ in \eqref{evd}, we can derive that
\begin{align*}
\int_{\mathbb{R}^N}|\nabla u_n|^2dx+\int_{\mathbb{R}^N}V_1(x)|u_n|^2dx-\mu_1\int_{\mathbb{R}^N}(I_{\alpha}\star|u_n|^{p})|u_n|^{p} dx
-\int_{\mathbb{R}^N}\beta u_nv_ndx=-\lambda_{1n}\xi^2.
\end{align*}
and
\begin{align*}
\int_{\mathbb{R}^N}|\nabla v_n|^2dx+\int_{\mathbb{R}^N}V_2(x)|v_n|^2dx
-\mu_2\int_{\mathbb{R}^N}(I_{\alpha}\star|v_n|^{q})|v_n|^{q} dx
-\int_{\mathbb{R}^N}\beta u_nv_ndx=-\lambda_{2n}\eta^2.
\end{align*}
Since $\{(u_n,v_n)\}$ is bounded in $H^1(\mathbb{R}^N)\times H^1(\mathbb{R}^N)$,  by \eqref{bu}, $B(u_n,p)$ and $B(v_n,q)$ are bounded. Therefore, $\lambda_{1n}$ and $\lambda_{2n}$ are bounded in $\mathbb{R}$. Moreover, there exist $\lambda_1\in\mathbb{R}$ and $\lambda_2\in\mathbb{R}$ such that $\lambda_{1n}\rightarrow\lambda_1$ and $\lambda_{2n}\rightarrow\lambda_2$. Therefore, $(u_0,v_0)$ satisfies
\begin{align*}
\left\{\begin{aligned}
&-\Delta u_0+V_1(x)u_0+\lambda_1 u_0=\mu_1(I_{\alpha}\star|u_0|^p)|u_0|^{p-2}u_0+\beta v_0,\\
&-\Delta v_0+V_2(x)v_0+\lambda_2 v_0=\mu_2(I_{\alpha}\star|v_0|^q)|v_0|^{q-2}v_0+\beta u_0,
\end{aligned}
\right.
\end{align*}
with $\|u_0\|_2^2=\xi_1^2$ and $\|v_0\|_2^2=\eta_1^2$. If $\xi_1=\xi$ and $\eta_1=\eta$, then the step is finished. Otherwise, $\xi_1^2+\eta_1^2<\xi^2+\eta^2$, i.e., there are three cases:
\begin{align*}
\left\{\begin{aligned}
 &\textbf{Case 1: }\xi_1<\xi,~\eta_1<\eta,\\
 &\textbf{Case 2: }\xi_1=\xi,~\eta_1<\eta,\\
 &\textbf{Case 3: }\xi_1<\xi,~\eta_1=\eta.
\end{aligned}
\right.
\end{align*}

By Lemma \ref{sub-lsplit},  $(\varphi_0,\psi_0)$ satisfies
\begin{align}\label{varphi0psi0}
\left\{\begin{aligned}
&-\Delta \varphi_0+\lambda_1 \varphi_0=\mu_1(I_{\alpha}\star|\varphi_0|^p)|\varphi_0|^{p-2}\varphi_0+\beta \psi_0,\\
&-\Delta \psi_0+\lambda_2 \psi_0=\mu_2(I_{\alpha}\star|\psi_0|^q)|\psi_0|^{q-2}\psi_0+\beta \varphi_0,
\end{aligned}
\right.
\end{align}
with $\|\varphi_0\|_2^2=\xi^2-\xi_1^2=\xi_2^2$ and $\|\psi_0\|_2^2=\eta^2-\eta_1^2=\eta_2^2$.

For \textbf{Case 2}, it is easy to find $\|\varphi_0\|_2^2=\xi^2-\xi_1^2=0$ and $\|\psi_0\|_2^2=\eta^2-\eta_1^2>0$, which implies that $\varphi_0=0$ and $\psi_0\neq0$. As a consequence, $(0,\psi_0)$ is a solution of \eqref{varphi0psi0}, which is impossible for $\beta>0$. Hence  \textbf{Case 2} doesn't hold. Similarly, \textbf{Case 3} also doesn't hold.

For \textbf{Case 1}, take $(\sqrt{u_0^2+\varphi_0^2},\sqrt{v_0^2+\phi_0^2})\in S_{\xi}\times S_{\eta}$ as the test function,  direct computation shows that
\begin{align}\label{sqrt2p}
\left\{\begin{aligned}
&|\nabla \sqrt{u_0^2+\varphi_0^2}|^2\le|\nabla u_0|^2+|\nabla\varphi_0|^2,~|\nabla \sqrt{v_0^2+\psi_0^2}|^2\le|\nabla v_0|^2+|\nabla\psi_0|^2,\\
&|\sqrt{u_0^2+\varphi_0^2}|^p\ge|u_0|^p+|\varphi_0|^p,~|\sqrt{v_0^2+\phi_0^2}|^q\ge|v_0|^q|\phi_0|^q,\\
&\beta\sqrt{u_0^2+\varphi_0^2}\sqrt{v_0^2+\phi_0^2}\ge \beta u_0v_0+\beta \varphi_0\psi_0\text{~for~}\beta>0.
 \end{aligned}
 \right.
\end{align}
From the second inequality in \eqref{sqrt2p}, we can derive that
\begin{align}\label{bup}
\begin{split}
B(\sqrt{u_0^2+\varphi_0^2},p)&=\int_{\mathbb{R}^N}\int_{\mathbb{R}^N}\frac{|\sqrt{u_0^2+\varphi_0^2}(x)|^p|\sqrt{u_0^2+\varphi_0^2}(y)|^p}{|x-y|^{N-\alpha}}dydx\\
&\ge\int_{\mathbb{R}^N}\int_{\mathbb{R}^N}\frac{(|u_0(x)|^p+|\varphi_0(x)|^p)(|u_0(y)|^p+|\varphi_0(y)|^p)}{|x-y|^{N-\alpha}}dydx\\
&\ge\int_{\mathbb{R}^N}\int_{\mathbb{R}^N}\frac{|u_0(x)|^p|u_0(y)|^p}{|x-y|^{N-\alpha}}dydx+\int_{\mathbb{R}^N}\int_{\mathbb{R}^N}\frac{|\varphi_0(x)|^p|\varphi_0(y)|^p}{|x-y|^{N-\alpha}}dydx\\
&=B(u_0,p)+B(\varphi_0,p),
\end{split}
\end{align}
similarly, there holds that
\begin{align}\label{bvq}
B(\sqrt{v_0^2+\psi_0^2},q)\ge B(v_0,q)+B(\psi_0,q).
\end{align}
By \eqref{sqrt2p}, \eqref{bup} and \eqref{bvq}, it yields that
\begin{align}\label{evtest}
\begin{split}
e_{V}(\xi,\eta)&\le E_{V}(\sqrt{u_0^2+\varphi_0^2},\sqrt{v_0^2+\psi_0^2})\\
&\le E_{V}(u_0,v_0)+E_{V}(\varphi_0,\psi_0)\\
&=E_{V}(u_0,v_0)+E(\varphi_0,\psi_0)+\frac12\int_{\mathbb{R}^N}[V_1(x)|\varphi_0|^2+V_2(x)|\psi_0|^2]dx.
\end{split}
\end{align}
Notice that
\begin{align}\label{V0}
\int_{\mathbb{R}^N}(V_1(x)|\varphi_0|^2+V_2(x)|\psi_0|^2)dx<0.
\end{align}
From Lemma \ref{sub-lsplit}, \eqref{evtest} and \eqref{V0}, it can be derived that
\begin{align*}
 e_{V}(\xi,\eta)&\le E_{V}(u_0,v_0)+E(\varphi_0,\psi_0)+\frac12\int_{\mathbb{R}^N}(V_1(x)|\varphi_0|^2+V_2(x)|\psi_0|^2)dx\\
 &<E_{V}(u_0,v_0)+E(\varphi_0,\psi_0)\\
 &=e_{V}(\xi_1,\eta_1)+e(\xi_2,\eta_2)\\
 &=e_{V}(\xi,\eta),
\end{align*}
which is impossible. Hence \textbf{Case 1} doesn't hold.  In conclusion, $\xi_1=\xi$ and $\eta_1=\eta$. Therefore, we complete the proof of \textbf{Step 1}.

\textbf{Step 2: $\{(u_n,v_n)\}$ is strongly convergent in $H^1(\mathbb{R}^N)\times H^1(\mathbb{R}^N)$}.

From \textbf{Step 1}, $(u_n,v_n)\rightarrow (u_0,v_0)$ strongly in $L^2(\mathbb{R}^N)\times L^2(\mathbb{R}^N)$. Hence $B(u_n,p)\rightarrow B(u_0,p)$ and $B(v_n,q)\rightarrow B(v_0,q)$  by \eqref{bu}. 

Notice that
$$
e_{V}(\xi,\eta)\le E_{V}(u_0,v_0)\le\liminf_{n}E_{V}(u_n,v_n)=e_{V}(\xi,\eta),
$$
which implies that
\begin{align*}
\|\nabla u_0\|_2^2+\|\nabla v_0\|_2^2
=\lim_{n\rightarrow\infty}\bigg\{\|\nabla u_n\|_2^2+\|\nabla v_n\|_2^2\bigg\}.
\end{align*}
Hence \textbf{Step 2} holds. Therefore, we establish the theorem.
\end{proof}

\subsection{Proof of Theorem \ref{tv1v2}}
In this subsection, we consider the case that $V_1(x)$ satisfies \textbf{(V1)} and $V_2(x)$ satisfies \textbf{(V2)}.  We work on the space $H^1(\mathbb{R}^N)\times \tilde{H}_2$.
\begin{lemma}\label{lscpt}
Assume that $V_2(x)$ satisfies \textbf{(V2)}. Then $\tilde{H}_2\hookrightarrow L^{p}(\mathbb{R}^N)(p\in[2,2^{*}))$ is compact.
\end{lemma}
Similar to Lemma \ref{lbfb}, we obtain that
\begin{align*}
E_{V}(u,v)\ge&\frac12(\|\nabla u\|_2^2+\|\nabla v\|_2^2)
-\frac{\mu_1C(N,\alpha,p)\xi^{2p(1-\delta_{p})}}{2p}\|\nabla u\|_2^{2p\delta_p}\\
&-\frac{\mu_2C(N,\alpha,q)\eta^{2q(1-\delta_{q})}}{2q}\|\nabla v\|_2^{2q\delta_q}
-\beta \xi\eta+\frac12\int_{\mathbb{R}^N}(V_1(x)|u|^2+V_2(x)|v|^2)dx.
\end{align*}
Since $\max\{1+\frac{\alpha}{N},2\}<p,q<1+\frac{\alpha+2}{N}$ and $(u,v)\in H^1(\mathbb{R}^N)\times\tilde{H_{2}}$, we have
$E_{V}(u,v)$ is also bounded from below. Moreover, there exists a bounded minimizing sequence $\{(u_n,v_n)\}\in H^1(\mathbb{R}^N)\times\tilde{H_{2}}$ with $\|u_n\|_2=\xi,~\|v_n\|_2=\eta$. Therefore, there exists a $(u_0,v_0)\in H^1(\mathbb{R}^N)\times\tilde{H_{2}}$ such that $(u_n,v_n)\rightharpoonup(u_0,v_0)$ weakly in $H^1(\mathbb{R}^N)\times\tilde{H_{2}}$.  Furthermore, by Lemma \ref{lscpt}, $\|v_0\|_2^2=\eta$. 
\begin{lemma}\label{lssub}
There holds that $e_{V}(\xi,\eta)\le e_{V}(\xi_1,\eta)+m(\xi-\xi_1,\mu_1)$, where $m(\xi-\xi_1,\mu_1)$ is defined in \eqref{mcnc}.
\end{lemma}
\begin{proof}
The proof is similar to Lemma \ref{evsubadd} (iii). We omit the details.
\end{proof}
Similar to Lemma \ref{sub-lsplit}, we give the following splitting lemma.
\begin{lemma}\label{lssplit}
Let $\{(u_n,v_n)\}\subset S_{\xi}\times S_{\eta}$ be a minimizing sequence for $e_{V}(\xi,\eta)$ such that $(u_n,v_n)\rightharpoonup (u_0,v_0)$ in $H^1(\mathbb{R}^N)\times H^1(\mathbb{R}^N) $ and let $\xi_1=\|u_0\|_2,~\eta=\|v_0\|_2$. If $\xi_1<\xi$, then there exist $\{y_n\}\subset\mathbb{R}^N$ and $w_0\in H^1(\mathbb{R}^N)\setminus{0}$ such that
\begin{align}\label{syninfty}
|y_n|\rightarrow\infty,~~u_n(\cdot+y_n)\rightharpoonup w_0\quad\text{in~}H^1(\mathbb{R}^N),
\end{align}
\begin{align}\label{sunu0}
\lim_{n\rightarrow\infty}\|u_n-u_0-w_0(\cdot-y_n)\|_2^2=0,
\end{align}
and $\xi^2=\xi_1^2+\xi_2^2$, where $\|w_0\|_2=\xi_2$. Furthermore,
\begin{align}\label{seveta}
e_{V}(\xi,\eta)=e_{V}(\xi_1,\eta)+m(\xi_2,\mu_1),
\end{align}
and
\begin{align}\label{su0}
E_{V}(u_0,v_0)=e_{V}(\xi_1,\eta),~E(w_0,0)=F_{\mu_1}(w_0)=m(\xi_2,\mu_1).
\end{align}
\end{lemma}
\begin{proof}
We divide the proof into three steps.

\textbf{Step 1: We find $\{y_{n}\}\subset\mathbb{R}^{N}$ and $w_0$ such that \eqref{syninfty} holds}.

We claim that
\begin{align*}
\delta_{0}:=\liminf _{n \rightarrow \infty} \sup _{y \in \mathbb{R}^{N}} \int_{y+B(0,1)}|u_{n}-u_{0}|^2dx>0,
\end{align*}
where $B(0,1):=\{x \in \mathbb{R}^{N}:|x|\le1\}$. We prove the claim by contradiction. In fact, we may assume that  $\delta_{0}=0$. Then by \eqref{bu}, $B(u_n,p)\to B(u_0,p)$. Hence
\begin{align*}
e_{V}(\xi_1,\eta)\le E_{V}(u_0,v_0)\le\liminf_{n}E_{V}(u_n,v_n)=e_{V}(\xi,\eta).
\end{align*}
By Lemma \ref{yt1} and Lemma \ref{lssub}, there holds that
\begin{align*}
e_{V}(\xi_1,\eta)\le e_{V}(\xi,\eta)\le e_{V}(\xi_1,\eta)+m(\xi-\xi_1,\mu_1)<e_{V}(\xi_1,\eta),
\end{align*}
it is impossible. Hence $\delta_{0}>0$.

Furthermore, by the fact $u_n\rightarrow u_0$ in $L_{loc}^{2}(\mathbb{R}^{N})$, we can find $\{y_n\} \subset \mathbb{R}^{N}$ such that
\begin{align*}
\int_{y_{n}+B(0,1)}|u_n-u_0|^{2}dx\rightarrow c_0>0~\text{ and}~~ |y_{n}| \rightarrow \infty.
\end{align*}
Denote that $u_n(\cdot+y_n)\rightarrow w_0$ weakly in $H^1(\mathbb{R}^N)$.  Due to $c_{0}>0$, we can deduce that $w_0\neq0$. Therefore, $\{y_{n}\}$ and $w_0$ satisfy \eqref{syninfty}.

Since $y_{n} \rightarrow \infty$, we have
\begin{align*}
&\|u_n-u_0-w_0(\cdot-y_n)\|_2^2\\
=&\|u_n\|_2^2+\|u_0\|_2^2+\|w_0\|_2^2-2\langle u_n, u_0\rangle_{L^2}-2\langle u_n(\cdot+y_n), w_0\rangle_{L^2}+o(1) \\
=&\|u_n\|_2^2-\|u_0\|_2^2-\|w_0\|_2^2+o(1),
\end{align*}
in other words,
\begin{align*}
\xi_2^2:=\|w_0\|_2^2\le\liminf_{n\rightarrow\infty}(\|u_n\|^2_2-\|u_0\|^2_2)=\xi^2-\xi_1^2,
\end{align*}
From $c_0>0$, it follows that $\xi_2>0$.

\textbf{Step 2: We show $\{y_{n}\} \subset \mathbb{R}^{N}$ and $w_0$ such that \eqref{sunu0} holds}.

Let $\delta_1^2=\lim\limits_{n\rightarrow\infty}\|u_n-u_0-w_0(\cdot-y_n)\|_2^2$. Then $\delta_1^2=\xi^2-\xi_1^2-\xi_2^2$. Next we show that $\delta_1=0$ by contradiction. Indeed, if $\delta_1>0$, then
\begin{align*}
&\|\nabla u_n\|_2^2-\|\nabla u_0\|_2^2-\|\nabla w_0\|_2^2-\|\nabla (u_n-u_0-w_0(\cdot-y_n))\|_2^2\\
=&-2\|\nabla u_0\|_2^2-2\|\nabla w_0\|_2^2+2\langle \nabla u_n, \nabla u_0\rangle+2\langle \nabla u_n(\cdot+y_n), \nabla w_0\rangle\\
=&o(1).
\end{align*}
Furthermore,
\begin{align*}
&\frac12\int_{\mathbb{R}^N}V_1(x)(|u_n|^2-| u_0|_2^2-|w_0|^2-|(u_n-u_0-w_0(\cdot-y_n))|^2)dx=o(1).
\end{align*}
In addition, by Br\'ezis-Lieb Lemma, there holds that
\begin{align*}
B(u_n,p)&=B(u_0,p)+B(u_n-u_0,p)+o(1)\\
&=B(u_0,p)+B(\varphi_0,p)+B(u _n-u_0-\varphi_0(\cdot-y_n),p)+o(1),
\end{align*}
and
$$
\int_{\mathbb{R}^N}\beta u_nv_n=\int_{\mathbb{R}^N}\beta (u_n-u_0)(v_n-v_0)+\int_{\mathbb{R}^N}\beta u_0v_0=o(1)+\int_{\mathbb{R}^N}\beta u_0v_0.
$$
Therefore,
\begin{align*}
E_{V}(u_n,v_n)
\ge & E_{V}(u_0,v_0)+F_{\mu_1}(w_0(\cdot-y_n))+\frac12\int_{\mathbb{R}^N}V_1(x)|w_0(\cdot-y_n)|^2dx\\
&+F_{\mu_1}(u_n-u_0-w_0(\cdot-y_n))+\frac12\int_{\mathbb{R}^N}V_1(x)|u_n-u_0-w_0(\cdot-y_n)|^2dx.
\end{align*}
Since  $|y_n|\rightarrow\infty,~\lim\limits_{|x|\rightarrow\infty}V_1(x)=0$ and $V_1(x)<0$, we can conclude that
\begin{align*}
\int_{\mathbb{R}^N}V_1(x)|w_0(x-y_n)|^2dx\rightarrow 0\text{~and}~\int_{\mathbb{R}^N}V_1(x)|u_n-u_0-w_0(x-y_n)|^2dx\rightarrow 0,
\end{align*}
as $n\rightarrow\infty$, which implies  that
\begin{align*}
E_{V}(u_n,v_n)
\ge E_{V}(u_0,v_0)+F_{\mu_1}(w_0(\cdot-y_n))
+F_{\mu_1}(u_n-u_0-w_0(\cdot-y_n).
\end{align*}
Hence
\begin{align*}
e_{V}(\xi,\eta)\ge e_{V}(\xi_1,\eta)+m(\xi_2,\mu_1)+m(\delta_1,\mu_1).
\end{align*}

By Lemma \ref{yt1} and Lemma \ref{lssub}, we can see that
\begin{align*}
e_{V}(\xi,\eta)&= e_{V}(\xi_1,\eta)+m(\xi_2,\mu_1)+m(\delta_1,\mu_1)\\
&> e_{V}(\xi_1,\eta)+m(\xi_2+\delta_1,\mu_1)\\
&\ge e_{V}(\xi_1+\xi_2+\delta_1,\eta)\\
&=e_{V}(\xi,\eta),
\end{align*}
which is a contradiction. Hence $\delta_1=0$.

\textbf{Step 3: We show $\{y_{n}\} \subset \mathbb{R}^{N}$ and $(\varphi_0,\psi_0)$ satisfy \eqref{seveta} and \eqref{su0}}.

Since $\delta_1=0$, we have
\begin{align*}
e_{V}(\xi,\eta)&=\lim_{n\rightarrow\infty}E_{V}(u_n,v_n)\\
&\ge\liminf_{n\rightarrow\infty}\bigg(E_{V}(u_0,v_0)+F_{\mu_1}(w_0(\cdot-y_n))+\frac12\int_{\mathbb{R}^N}V_1(x)|w_0(\cdot-y_n)|^2dx\bigg)\\
&\ge E_{V}(u_0,v_0)+F_{\mu_1}(w_0)\\
&\ge e_{V}(\xi_1,\eta)+m(\xi_2,\mu_1),
\end{align*}
together with  Lemma \ref{lssub}, we can obtain that $e_{V}(\xi,\eta)=e_{V}(\xi_1,\eta)+m(\xi_2,\mu_1)$ and $E_{V}(u_0,v_0)=e_{V}(\xi_1,\eta),~F_{\mu_1}(w_0)=m(\xi_2,\mu_1)$. Thus we complete the proof.
\end{proof}

In the following, we give the $L^2-$convergence of $\{u_n\}$.
\begin{lemma}
$\|u_0\|_2=\xi$.
\end{lemma}
\begin{proof}
Take $(\sqrt{u_0^2+w_0^2},v_0)\in S_{\xi}\times S_{\eta}$ as the test function, where $w_0$ is defined as in Lemma \ref{lssplit}. From $\beta\int_{\mathbb{R}^N}\sqrt{u_0^2+w_0^2}v_0dx\ge\beta\int_{\mathbb{R}^N}u_0v_0dx$ for $w_0\neq0$ and \eqref{bup}, we can derive that
\begin{align}\label{sevs}
\begin{split}
e_{V}(\xi,\eta)&\le E_{V}(\sqrt{u_0^2+w_0^2},v_0)\\
&\le E_{V}(u_0,v_0)+\frac12\|\nabla w_0\|_2^2-\frac{\mu_1}{2p}B(w_0,p)+\frac12\int_{\mathbb{R}^N}V_1(x)w_0^2dx,\\
\end{split}
\end{align}
due to $V_1(x)<0$, it means that
\begin{align*}
\int_{\mathbb{R}^N}V_1(x)w_0^2dx<0,
\end{align*}
together  with \eqref{sevs}, by Lemma \ref{lssplit}, it follows that
\begin{align*}
e_{V}(\xi,\eta)&< E_{V}(u_0,v_0)+\frac{1}{2}\|\nabla w_0\|_2^2-\frac{\mu_1}{2p}B(w_0,p)\\
&=e_{V}(\xi_1,\eta)+F_{\mu_1}(w_0)\\
&=e_{V}(\xi_1,\eta)+m(\xi_2,\mu_1)\\
&=e_{V}(\xi,\eta),
\end{align*}
which is a contradiction. Hence $\|u_0\|_2=\xi$. In conclusion,  $(u_n,v_n)\rightarrow(u_0,v_0)$ strongly in $L^2(\mathbb{R}^N)\times L^2(\mathbb{R}^N)$.
\end{proof}
Similar to the proof of Theorem \ref{t1}, we prove Theorem \ref{tv1v2}.
\begin{proof}[\textbf{Proof of Theorem \ref{tv1v2}}]
By the facts
\begin{align*}
\left\{\begin{aligned}
&(u_n,v_n)\rightharpoonup(u_0,v_0)\text{~weakly in~} H^1(\mathbb{R}^N)\times H^1(\mathbb{R}^N),\\ &(u_n,v_n)\rightarrow(u_0,v_0)\text{~strongly in~} L^2(\mathbb{R}^N)\times L^2(\mathbb{R}^N),
\end{aligned}
\right.
\end{align*}
and \eqref{bu}, it yields that
$B(u_n,p)\rightarrow B(u_0,p)$ and $B(v_n,q)\rightarrow B(v_0,q)$. Furthermore, it can be deduced that $(u_n,v_n)\rightarrow(u_0,v_0)$ strongly in $H^1(\mathbb{R}^N)\times H^1(\mathbb{R}^N)$. Hence we finish the proof.
\end{proof}

\end{document}